\documentclass[10pt,reqno]{amsart}
\usepackage{pdfsync}
\usepackage{amssymb,amscd}  
\usepackage{color} 

\newtheorem{theorem}{Theorem}
\newtheorem{proposition}[theorem]{Proposition}            
\newtheorem{corollary}[theorem]{Corollary}                
\newtheorem{lemma}[theorem]{Lemma}
\newtheorem{remark}[theorem]{Remark}

\newtheorem{definition}{Definition}

\begin{document}
\title[Discrete Cores of Free Products]{Discrete Cores of type III free product factors}
\author[Y. Ueda]
{Yoshimichi UEDA}
\address{
Graduate School of Mathematics, 
Kyushu University, 
Fukuoka, 819-0395, Japan
}
\email{ueda@math.kyushu-u.ac.jp}
\thanks{Supported by Grant-in-Aid for Scientific Research (C) 24540214.}
\thanks{AMS subject classification: Primary:\, 46L54;
secondary:\,46L10.}
\thanks{Keywords: Type III factor, Free product, Discrete decomposition, Hyperfinite von Neumann algebra, Free group factor, Almost periodic state.}

\begin{abstract} 
We give a general description of the discrete decompositions of type III factors arising as central summands of free product von Neumann algebras based on our previous works. This enables us to give several precise structural results on type III free product factors. 
\end{abstract} 
\maketitle

\allowdisplaybreaks{

\section{Introduction} 

The present paper is a sequel to our previous studies of completely general free product von Neumann algebras with emphasis on type III factors \cite{Ueda:AdvMath11},\cite{Ueda:MRL11}. The main purpose here is to give one simple description of the discrete decomposition of arbitrary type III factor arising as (central summand of) a free product von Neumann algebra in terms of amalgamated free product of very particular form and its applications. The description enables us to complete our `initial' research plan to study free product von Neumann algebras from the viewpoint of general structure theory of type III factors mainly due to Connes and Takesaki. The plan dates back to the mid 90s; in fact, our initial attempts in the direction are \cite{Ueda:PacificJMath99},\cite{Ueda:MathScand01}. The consequences that we will give as applications of the description, which will be summarized in the next paragraph, are precise structural results that certainly have some potential in future applications of free product von Neumann algebras. 

Firstly the description provides an alternative (and much more natural) route toward the question that Dykema seriously investigated in \cite{Dykema:FieldsInstituteComm97}, with avoiding some difficult points there, and moreover enables us to improve his previous result by removing all the superfluous assumptions. In particular, our result provides a natural class of von Neumann algebras which includes all the hyperfinite von Neumann algebras and is closed under taking free products with respect to arbitrary almost periodic states (including tracial ones and periodic ones by definition). See Theorem \ref{T3} and Corollary \ref{C4} for precise statements. Remark here that we have already known, by \cite[Corollary 3.1]{Ueda:MRL11}, that the almost periodicity of given states is necessary and sufficient to make the type III factor that appears as a central summand of the resulting free product von Neumann algebra admits discrete decomposition. A specialization of the result is that the centralizer of the free product state of any `non-tracial' free product of hyperfinite von Neumann algebras with respect to almost periodic states becomes an irreducible subfactor isomorphic to $L(\mathbb{F}_\infty)$ (in the diffuse factor summand; remark that any non-trivial and non-tracial free product von Neumann algebra is a diffuse factor plus a finite dimensional algebra, see \cite[Theorem 4.1]{Ueda:AdvMath11}). This is nothing less than positive evidence to the question \cite[\S\S5.4]{Ueda:AdvMath11} asking whether or not any free product von Neumann algebra of hyperfinite ones is a free Araki--Woods factor introduced by Shlyakhtenko \cite{Shlyakhtenko:PacificJMath97} (plus a finite dimensional algebra). 

Secondly we point out that all the results mentioned so far and a recent work due to Boutonnet, Houdayer and Raum \cite{BoutonnetHoudayerRaum:Preprint12} enable one to show the lack of Cartan subalgebra in both the centralizer and the discrete core of the (unique) type III factor summand of an arbitrary free product von Neumann algebra with respect to almost periodic states; see Proposition \ref{P6} for details. Related to this we also point out that both the centralizer and the discrete core of the (unique) type III factor summand of an arbitrary free product von Neumann algebra with respect to almost periodic states must be prime; see Proposition \ref{P7}.    

Thirdly we derive, from the above-mentioned description of discrete cores, two general structural results on the centralizers of free product states of particular form. See Proposition \ref{P9} for precise statements. Part of these results enables us to relate \cite[Corollary 6.5 (4)]{IoanaPetersonPopa:Acta08} to \cite[Theorem 5.10]{Houdayer:Crelle09} via discrete decompositions of type III factors. See Remark \ref{R10} for details.  

We follow the notational rule in our previous papers \cite{Ueda:AdvMath11},\cite{Ueda:MRL11} (see the glossary at the end of the introduction of \cite{Ueda:AdvMath11}). In what follows, we say that a von Neumann algebra is non-trivial if it is not $1$ dimensional. In the next \S2 we give and prove the main results. The next \S3 provides some addition to \cite{Dykema:MathProcCambPhilSoc04}, which is necessary in \S\S2.5. The final \S4 is less original, but makes the contents in \S\S2.3 be accessible by any serious reader. In fact, some discussions in \cite{Redelmeier:arXiv:1207.1117} seem sketchy, and the core discussions there and in its source \cite{DykemaRedelmeier:arXiv11} look needlessly involved for the necessary fact here, because those works aimed at stronger assertions. Hence we include a simplified (or `optimized') proof of the fact that we need in \S\S2.3 for the reader's convenience.  

\section{Main Results} 

\subsection{Preliminaries}\label{SS2.1} Throughout this paper, we assume that {\it $M_i$, $i=1,2$, are non-trivial $\sigma$-finite von Neumann algebras with $(\dim(M_1),\dim(M_2)) \neq (2,2)$ and $\varphi_i$, $i=1,2$, are faithful normal states on them, respectively}, unless otherwise specified. Here are the consequences of \cite[Theorem 4.1]{Ueda:AdvMath11} with a useful remark that immediately comes from its proof. The resulting free product $(M,\varphi) = (M_1,\varphi)\star(M_2,\varphi_2)$ has the following structure: $M = M_d\oplus M_c$, where $M_d$ is a finite dimensional algebra (that can be described explicitly) and $M_c$ a factor of type II$_1$ or III$_\lambda$ ($\lambda\neq0$) such that the T-set $T(M_c) = \{t \in \mathbb{R}\,|\, \sigma_t^{\varphi_1} = \mathrm{Id} = \sigma_t^{\varphi_2}\}$ and $M_c' \cap M_c^\omega = \mathbb{C}$, where $M_c^\omega$ denotes the ultraproduct von Neumann algebra of $M_c$ associated with free ultrafilter $\omega$ (see e.g.~\cite[\S\S2.2]{Ueda:AdvMath11}). In particular, $M_c$ is of type II$_1$ if and only if both the $\varphi_i$ are tracial. When the finite dimensional algebra summand $M_d$ is appeared, the proof of the theorem implicitly shows the following (which is probably important in actual applications): there are a unique $i_0 \in \{1,2\}$ and a unique non-zero projection $p \in \mathcal{Z}(M_{i_0})$ so that $M_{i_0}p = \mathbb{C}p$, $1_{M_d} \leq p$ (in $M$), and hence $M_c$ is stably isomorphic to $p^\perp M p^\perp$ with $p^\perp := 1-p$.  Here Dykema's observation (see e.g.~\cite[Lemma 2.1]{Ueda:AdvMath11}) shows that the pair $(p^\perp M p^\perp,\delta\varphi\!\upharpoonright_{p^\perp M p^\perp})$ with $\delta := \varphi(p^\perp)^{-1}$ becomes either 
\begin{equation}\label{Eq1}
(M_1 p^\perp, \delta\varphi_1\!\upharpoonright_{M_1 p^\perp}) \star 
(p^\perp Np^\perp, \delta\varphi\!\upharpoonright_{p^\perp Np^\perp}) \quad \text{or} \quad 
(p^\perp Np^\perp, \delta\varphi\!\upharpoonright_{p^\perp Np^\perp}) \star (M_2 p^\perp, \delta\varphi_1\!\upharpoonright_{M_2 p^\perp}), 
\end{equation}
where the pair $(N,\varphi\!\upharpoonright_N)$ is 
\begin{equation}\label{Eq2}
(\mathbb{C}p\oplus\mathbb{C}p^\perp,\varphi_1\!\upharpoonright_{\mathbb{C}p\oplus\mathbb{C}p^\perp})\star(M_2,\varphi_2) \quad \text{or} \quad  
(M_1,\varphi_1)\star(\mathbb{C}p\oplus\mathbb{C}p^\perp,\varphi_2\!\upharpoonright_{\mathbb{C}p\oplus\mathbb{C}p^\perp}),
\end{equation}
respectively. Consequently, the diffuse factor summand $M_c$ is again a free product von Neumann algebra up to stable isomorphism. Hence the analysis of $M_c$ comes down to that on some free product factor with care of the descriptions Eq.\eqref{Eq1} and Eq.\eqref{Eq2}. Moreover, if a non-zero projection $e \in \mathcal{Z}(M_i)$ ($i=1$ or $2$) satisfies that $M_i e$ is diffuse, then $e \leq 1_{M_c}$, implying that $M_c$ is stably isomorphic to $eMe$, and the pair $(eMe,\varphi(e)^{-1}\varphi\!\upharpoonright_{eMe})$ is the free product of $(M_i,\varphi_i(e)^{-1}\varphi_i\!\upharpoonright_{M_i e})$ and non-trivial something, like Eq.\eqref{Eq1}-\eqref{Eq2}.   

\medskip
In the rest of this paper we further assume that both the $M_i$'s have separable preduals. 

\subsection{Discrete decomposition associated with free products}\label{SS2.2}
In what follows we assume that the diffuse factor summand $M_c$ is of type III, that is, at least one of the $\varphi_i$'s is not tracial as explained in \S\S2.1. The following facts were established in \cite[Theorem 2.1, Corollary 3.1]{Ueda:MRL11}: The diffuse factor summand $M_c$ possesses an almost periodic state if and only if both the $\varphi_i$ are almost periodic. (Recall that a faithful normal positive linear functional is said to be almost periodic if its modular operator is diagonalizable; hence tracial ones and  periodic ones are particular cases.) When this is the case, the restriction $\varphi_c := \varphi\!\upharpoonright_{M_c}$ becomes almost periodic again and $((M_c)_{\varphi_c})'\cap M_c^\omega = \mathbb{C}$; in particular, $\varphi_c$ is `extremal', which means that its centralizer becomes an irreducible subfactor of $M_c$. Hence, if the full type III factor $M_c$ admits discrete decomposition in the sense of Connes \cite{Connes:JFA74}, then the decomposition can be described as follows. In the case, both the given $\varphi_i$ must be almost periodic, and the Sd-invariant $\Gamma := \mathrm{Sd}(M_c)$ of $M_c$  becomes the multiplicative subgroup of $\mathbb{R}_+^\times$ algebraically generated by the point spectra of the modular operators $\Delta_{\varphi_i}$. Denote by $\beta$ the embedding of $\Gamma$ to $\mathbb{R}_+^\times$ and equip $\Gamma$ with the discrete topology. By group duality one gets a continuous group homomorphism $\hat{\beta} : \mathbb{R} \to G := \widehat{\Gamma}$ with dense range so that $\langle \gamma, \hat{\beta}(t) \rangle = \gamma^{\sqrt{-1}t}$ for $\gamma \in \Gamma$ and $t \in \mathbb{R}$, which gives the continuous group homomorphism $g \in G \mapsto \sigma_g^{\psi,\Gamma}$ into the automorphism group of $M$ or $M_i$ with $\psi := \varphi$ or $\varphi_i$, respectively, so that $\sigma_{\hat{\beta}(t)}^{\psi,\Gamma} = \sigma_t^{\psi}$ for every $t \in \mathbb{R}$. Since $\hat{\beta}(\mathbb{R})$ is dense in $G$, $\psi\circ\sigma_g^{\psi,\Gamma} = \psi$ holds for every $g \in G$ and each $\sigma_g^{\psi,\Gamma}$ fixes any element of the center. See e.g.~\cite[Proposition 1.1]{Connes:JFA74}. 
The {\it discrete decomposition} is 
\begin{equation}\label{Eq3}
M_c \cong \big(M_c\rtimes_{\sigma^{\varphi,\Gamma}}G\big)\rtimes_{\theta}\Gamma
\end{equation} 
with $\theta := \widehat{\sigma^{\varphi,\Gamma}}$, the dual action of $\sigma^{\varphi,\Gamma}$, which clearly comes from the cerebrated Takesaki duality (see \cite[Theorem X.2.3]{Takesaki:Book}),   
and the first crossed product $\widehat{M_c} := M_c\rtimes_{\sigma^{\varphi,\Gamma}}G$ is usually called the {\it discrete core} of $M_c$. The decomposition is unique, see \cite[Corollary 4.12]{Connes:JFA74}. (Remark that our formulation is a little bit different from Connes's, see Theorem \ref{T1} (1.4) below, but no essential difference occurs.) It turns out that the centralizer $(M_c)_{\varphi_c}$ is stably isomorphic to the discrete core, that is, 
\begin{equation}\label{Eq4}
(M_c)_{\varphi_c}\,\bar{\otimes}\,B(\ell^2) \cong M_c\rtimes_{\sigma^{\varphi,\Gamma}}G 
\end{equation}
since $\sigma^{\varphi,\Gamma}$ is a minimal action (see e.g.~the proof of Lemma \ref{L5} below). Hence the centralizer $(M_c)_{\varphi_c}$ is captured by the discrete core of $M_c$. 

\medskip
With these preliminary facts we give a general description of $M\rtimes_{\sigma^{\varphi,\Gamma}}G$ instead of $\widehat{M_c}$ itself. It is obviously a discrete decomposition analogue of our old result \cite[Theorem 5.1]{Ueda:PacificJMath99} with extra structural fact (1.3) that follows from \cite[Theorem 4.1]{Ueda:AdvMath11} and \cite[Theorem 2.1]{Ueda:MRL11}.    

\begin{theorem}\label{T1} Under the assumptions above we have 
\begin{equation}\label{Eq5}
(M\rtimes_{\sigma^{\varphi,\Gamma}}G,E_\varphi) = (M_1\rtimes_{\sigma^{\varphi_1,\Gamma}}G,E_{\varphi_1}) \star_{\mathbb{C}\rtimes G}(M_2\rtimes_{\sigma^{\varphi_2,\Gamma}}G,E_{\varphi_2}), 
\end{equation}
where the conditional expectation $E_{\varphi}$ is the restriction of the Fubini map $\varphi\bar{\otimes}\mathrm{Id}$ to the crossed product $M\rtimes_{\sigma^{\varphi,\Gamma}}G$ $(\subseteq M\bar{\otimes}B(L^2(G))$ and the $E_{\varphi_i}$, $i=1,2$, are defined just as the restrictions of $E_\varphi$ to $M_i\rtimes_{\sigma^{\varphi_i,\Gamma}}G$ $(\subseteq M\rtimes_{\sigma^{\varphi,\Gamma}}G)$. Moreover, we have{\rm:} 
\begin{itemize}
\item[(1.1)] $\mathbb{C}\rtimes G = L(G) \cong \ell^\infty(\Gamma)$ by $e_\gamma = \int_G \overline{\langle\gamma,g\rangle}\lambda(g)\,dg \longleftrightarrow \delta_\gamma$ for each $\gamma \in \Gamma$ with the Haar probability measure $dg$, where the $\delta_\gamma$, $\gamma \in \Gamma$, are the functions in $\ell^\infty(\Gamma)$ defined to be $\delta_\gamma(\gamma') := \delta_{\gamma,\gamma'}$, $\gamma' \in \Gamma$.  
\item[(1.2)] $M\rtimes_{\sigma^{\varphi,\Gamma}}G$ is semifinite and possesses a unique faithful normal semifinite tracial weight $\mathrm{Tr}$ that satisfies $\mathrm{Tr}\circ E_\varphi = \mathrm{Tr}$ and $\mathrm{Tr}(e_\gamma) = \gamma^{-1}$ for every $\gamma \in \Gamma$. 
\item[(1.3)] The central decomposition of $M\rtimes_{\sigma^{\varphi,\Gamma}}G$ becomes the direct sum of countably infinite copies of $M_d$ and the discrete core $\widehat{M_c}$. 
\item[(1.4)] The dual action $\theta$ comes from an action on the smallest algebra $\mathbb{C}\rtimes G$, which satisfies that $\theta_\gamma(e_{\gamma'}) = e_{\gamma\gamma'}$ for every $\gamma,\gamma' \in \Gamma$ and $\mathrm{Tr}\circ\theta_\gamma = \gamma^{-1}\mathrm{Tr}$ for every $\gamma \in \Gamma$ .    
\end{itemize}
\end{theorem}
\begin{proof} Firstly note that $\varphi\circ\sigma^{\varphi,\Gamma}_g = \varphi$ for every $g \in G$ as remarked before. Hence the conditional expectation $E_\varphi : M\rtimes_{\sigma^{\varphi,\Gamma}}G \to \mathbb{C}\rtimes G$ makes sense. Then   Eq.\eqref{Eq5} follows from the almost trivial fact that $M_1\bar{\otimes}B(L^2(G))$ and $M_2\bar{\otimes}B(L^2(G))$ are freely independent with amalgamation over $\mathbb{C}\bar{\otimes}B(L^2(G))$ with respect to the Fubini map $\varphi\bar{\otimes}\mathrm{Id}$, see \cite[Proposition 1]{Ueda:IMRN00}. 

The assertion (1.1) is standard. 

Let $\widetilde{\varphi}$ be the dual weight on $M\rtimes_{\sigma^{\varphi,\Gamma}}G$ constructed out of $\varphi$. See \cite[Definition X.1.16]{Takesaki:Book}. Let us briefly review its construction for the reader's convenience. Let $K(G,M)$ be the involutive algebra of $\sigma$-strong*-continuous functions from $G$ to $M$ endowed with product and involution 
$$
(x*y)(g) = \int_G \sigma_{g'}^{\varphi,\Gamma}(x(gg')) y(g'{}^{-1})\,dg', \quad x^\sharp(g) = \sigma_{g^{-1}}^{\varphi,\Gamma}(x(g^{-1}))^*. 
$$
The mapping $x \in K(G,M) \mapsto \int_G \lambda(g)\, x(g)\,dg \in M\rtimes_{\sigma^{\varphi,\Gamma}}G$ defines a $*$-representation, where the $\lambda(g)$ are the canonical unitary representation of $G$ into $M\rtimes_{\sigma^{\varphi,\Gamma}}G$. By \cite[Theorem X.1.17, Eq.(42)]{Takesaki:Book} together with the usual polarization trick one has
\begin{equation*} 
\widetilde{\varphi}((\sideset{}{_G}\int \lambda(g)\,y(g)\,dg)^*(\sideset{}{_G}\int \lambda(g)\,x(g)\,dg)) = \int_G \varphi(y(g)^* x(g))\,dg, \quad x,y \in K(G,M). 
\end{equation*} 
In particular, $\widetilde{\varphi}(e_\gamma) = \tilde{\varphi}(e_\gamma^* e_\gamma) = \int_G \varphi(1)\,dg = \varphi(1)$, which shows that the restriction of $\widehat{\varphi}$ to $\mathbb{C}\rtimes G$ is semifinite. 

Since $\varphi\circ\sigma^{\varphi,\Gamma}_g = \varphi$ for every $g \in G$, we have, by \cite[Theorem X.1.17]{Takesaki:Book}, $\sigma_t^{\widetilde{\varphi}}(x) = \sigma_t^\varphi(x)$ for every $x \in M \subseteq M\rtimes_{\sigma^{\varphi,\Gamma}}G$ and $\sigma_t^{\widetilde{\varphi}}(\lambda(g)) = \lambda(g)$ for every $g \in G$. In particular, we get $\sigma_t^{\widetilde{\varphi}} = \mathrm{Ad}\lambda(\hat{\beta}(t))$ for every $t \in \mathbb{R}$. This means that the unique non-singular positive self-adjoint operator $H$ affiliated with $\mathbb{C}\rtimes G$ determined by $\lambda(\hat{\beta}(t)) = H^{\sqrt{-1}t}$, $t \in \mathbb{R}$, gives a faithful normal semifinite tracial weight $\mathrm{Tr}(-) := \widetilde{\varphi}_{H^{-1}}(-) = \lim_{\varepsilon\searrow0}\widetilde{\varphi}((H^{-1/2})_\varepsilon\,(-)\,(H^{-1/2})_\varepsilon)$ with $(H^{-1/2})_\varepsilon := H^{-1/2}(1+\varepsilon H^{-1/2})^{-1}$, $\varepsilon > 0$, see \cite[Lemma VIII.2.8]{Takesaki:Book}. The above formula of $\sigma_t^{\widetilde{\varphi}}$ together with the semifiniteness of the restriction $\widetilde{\varphi}\!\upharpoonright_{\mathbb{C}\rtimes G}$ shows the existence of a faithful normal conditional expectation $E: M\rtimes_{\sigma^{\varphi,\Gamma}}G \to \mathbb{C}\rtimes G$. For every $x \in M \subseteq M\rtimes_{\sigma^{\varphi,\Gamma}}G$ and every $\gamma_1, \gamma_2 \in \Gamma$ one has $e_{\gamma_1}^* E(x)e_{\gamma_2} = E(e_{\gamma_1}^* x e_{\gamma_2})$ and $xe_{\gamma_2} = \int_G \lambda(g) (\overline{\langle \gamma_2, g\rangle} \sigma_{g^{-1}}^{\varphi,\Gamma}(x))\,dg$, and hence the usual polarization trick justifies that  
\begin{align*}
\widetilde{\varphi}(e_{\gamma_1}^* E(x) e_{\gamma_2}) 
&= \widetilde{\varphi}\big(\big(
\int_G \lambda(g) \overline{\langle \gamma_1, g\rangle}\,dg\big)^* 
\big(\int_G \lambda(g) (\overline{\langle \gamma_2, g\rangle} \sigma_{g^{-1}}^{\varphi,\Gamma}(x))\,dg\big)\big) \\
&= 
\int_G \langle \gamma_1, g\rangle\,\overline{\langle \gamma_2, g\rangle} \varphi(\sigma_{g^{-1}}^{\varphi,\Gamma}(x))\,dg 
= \varphi(x)\widetilde{\varphi}(e_{\gamma_1}^* e_{\gamma_2}).
\end{align*} 
Once passing to the GNS representation associated with the restriction $\widetilde{\varphi}\!\upharpoonright_{\mathbb{C}\rtimes G}$ one can conclude $E(x) = \varphi(x)1 = E_\varphi(x)$, since the linear span of the $e_\gamma$, $\gamma \in\Gamma$ form a $\sigma$-strong-dense $*$-subalgebra of $\mathbb{C}\rtimes G$ so that it is not hard to see that the linear span becomes a dense subspace of the GNS Hilbert space via the canonical embedding. It immediately follows that $E = E_\varphi$, and hence we see that $E_\varphi = E$ preserves the $\mathrm{Tr}$, since $(H^{-1/2})_\varepsilon \in \mathbb{C}\rtimes G$ for every $\varepsilon>0$. Notice that 
$$
\lambda(g) e_\gamma = \int_G \overline{\langle\gamma,g'\rangle}\lambda(gg')\,dg' = \int_G \overline{\langle\gamma,g^{-1}g'\rangle}\lambda(g')\,dg' = \langle\gamma,g\rangle e_\gamma,
$$
implying $\lambda(g) = \sum_{\gamma\in\Gamma} \langle\gamma,g\rangle e_\gamma$ for every $g \in G$. Thus we have
$$
H^{\sqrt{-1}t} = \lambda(\hat{\beta}(t)) = \sum_{\gamma\in\Gamma}\langle\gamma,\hat{\beta}(t)\rangle e_\gamma = \sum_{\gamma\in\Gamma} \gamma^{\sqrt{-1}t} e_\gamma, 
$$ 
and the spectral decomposition $H = \sum_{\gamma\in\Gamma}\gamma\,e_\gamma$. Since $(H^{-1/2})_\varepsilon\,e_\gamma\,(H^{-1/2})_\varepsilon = \frac{\gamma^{-1}}{(1+\varepsilon\gamma^{-1/2})^2} e_\gamma$ for each $\varepsilon>0$, we get 
\begin{align*}
\mathrm{Tr}(e_\gamma) &= \lim_{\varepsilon\searrow0}\widetilde{\varphi}((H^{-1/2})_\varepsilon\,e_\gamma\,(H^{-1/2})_\varepsilon)
= \lim_{\varepsilon\searrow0}\frac{\gamma^{-1}}{(1+\varepsilon\gamma^{-1/2})^2}\widetilde{\varphi}(e_\gamma) = \gamma^{-1}\varphi(1) = \gamma^{-1}.
\end{align*} 
Obviously the $\mathrm{Tr}$ is uniquely determined by $\mathrm{Tr}\circ E_\varphi = \mathrm{Tr}$ and $\mathrm{Tr}\!\upharpoonright_{\mathbb{C}\rtimes G}$. Hence we have obtained the assertion (1.2). 

The assertion (1.3) is seen as follows. We have 
$$
M\rtimes_{\sigma^{\varphi,\Gamma}}G = (M_d\rtimes_{\sigma^{\varphi,\Gamma}}G)\oplus(M_c\rtimes_{\sigma^{\varphi,\Gamma}}G) \cong (M_d\,\bar{\otimes}\,\ell^\infty(\Gamma))\oplus(M_c\rtimes_{\sigma^{\varphi,\Gamma}}G), 
$$
where the last isomorphism comes from \cite[Theorem X.1.7 (ii)]{Takesaki:Book}, since $\sigma_g^{\varphi,\Gamma}$ fixes any element in $\mathcal{Z}(M)$ and hence its restriction to $M_d$ must be inner for every $g \in G$ (see e.g.~\cite[Exercise V.1.4]{Takesaki:Book}). Then $M_d\,\bar{\otimes}\,\ell^\infty(\Gamma)$ is a direct sum of countably infinite copies of $M_d$, and $M_c\rtimes_{\sigma^{\varphi,\Gamma}}G$ is already known to be the discrete core of the diffuse factor summand $M_c$.  

Lastly we prove the assertion (1.4). The dual action $\theta$ is constructed by $\theta_\gamma(\lambda(g)) = \overline{\langle\gamma, g\rangle}\lambda(g)$ for every $\gamma \in \Gamma$ and every $g \in G = \widehat{\Gamma}$. As above 
$\lambda(g) = \sum_{\gamma' \in \Gamma} \langle \gamma', g\rangle\,e_{\gamma'}$ so that for arbitrary $\gamma,\gamma'\in\Gamma$ one has 
$$
\theta_\gamma(\lambda(g)) = \overline{\langle \gamma,g\rangle} \sum_{\gamma' \in \Gamma} \langle \gamma', g\rangle\,e_{\gamma'} = \sum_{\gamma' \in \Gamma} \langle \gamma^{-1}\gamma', g\rangle\,e_{\gamma'} = \sum_{\gamma' \in \Gamma} \langle\gamma',g\rangle\,e_{\gamma\gamma'}, 
$$  
implying $\theta_\gamma(e_{\gamma'}) = e_{\gamma\gamma'}$ by the uniqueness of spectral decomposition. It then follows that $\mathrm{Tr}\circ\theta_\gamma(e_{\gamma'}) = \mathrm{Tr}(e_{\gamma\gamma'}) = \gamma^{-1}\gamma'^{-1} = \gamma^{-1}\mathrm{Tr}(e_{\gamma'})$. Hence $\mathrm{Tr}\circ\theta_\gamma = \gamma^{-1}\mathrm{Tr}$ for every $\gamma \in \Gamma$, since $\theta_\gamma\!\upharpoonright_M = \mathrm{Id}_M$. 
\end{proof}

\subsection{First precise structural results}\label{SS2.3}  
Theorem \ref{T1} tells us that the analysis of the discrete core $\widehat{M_c}$ or equivalently the centralizer $(M_c)_{\varphi_c}$ essentially comes down to that of certain semifinite tracial amalgamated free products over atomic abelian von Neumann algebras. This is very fruitful; in fact, we will be able to prove the completion of \cite[Theorem 3]{Dykema:FieldsInstituteComm97}. The proof below is completely independent of \cite{Dykema:FieldsInstituteComm97}, but needs a recent attempt due to Redelmeier \cite{Redelmeier:arXiv:1207.1117} to generalize the previous works due to Dykema partly with him \cite{Dykema:DukeMathJ93},\cite{Dykema:AmerJMath95},\cite{Dykema:BLMS11},\cite{DykemaRedelmeier:arXiv11} on finite tracial (amalgamated) free products based on free probability machinery to the semifinite setting. The algorithm to get explicit algebraic structure in \cite{Redelmeier:arXiv:1207.1117} is unfortunately not so clear at all (though he tried to illustrate it by several explicit examples), but what we actually need here is only the next weak version of Redelmeier's assertion, since we have already known some details on the structure of $M\rtimes_{\sigma^{\sigma,\Gamma}}G$.

\begin{proposition}\label{P2} 
The class, denoted by $\mathcal{R}_4$, consisting of all countable direct sums of semifinite hyperfinite von Neumann algebras with separable preduals and {\rm(}at most countably infinite{\rm)} amplifications of interpolated free group factors is closed under taking semifinite tracial amalgamated free products over atomic type I von Neumann subalgebras. 
\end{proposition} 

Only this assertion is not so hard to prove except some type II$_\infty$ problems, see \S4. Here is the main result of this subsection. 
 
\begin{theorem}\label{T3} Let $M_i$, $i=1,2$, be non-trivial von Neumann algebras with separable preduals and $\varphi_i$, $i=1,2$, be faithful normal states on them, respectively. Assume that at least one of the $\varphi_i$'s is not tracial, and further that each $(M_i,\varphi_i)$ is either 
\begin{itemize} 
\item[{\rm(i)}] $M_i$ is hyperfinite and $\varphi_i$ almost periodic{\rm;} 
\item[{\rm(ii)}] $M_i$ is an amplification of an interpolated free group factor and $\varphi_i$ almost periodic{\rm;}  
\item[{\rm(iii)}] $M_i$ is a full type III factor with admitting discrete decomposition whose discrete core is stably isomorphic to $L(\mathbb{F}_\infty)$, and $\varphi_i$ almost periodic{\rm;} or 
\item[{\rm(iv)}] a countable direct sum of pairs from {\rm(i)}--{\rm(iii)}. 
\end{itemize} 
Then the unique diffuse factor summand $M_c$ of the free product $(M,\varphi) = (M_1,\varphi_1)\star(M_2,\varphi_2)$ is of type III and admits discrete decomposition whose discrete core $\widehat{M_c}$ is stably isomorphic to $L(\mathbb{F}_\infty)$. Moreover, the centralizer $(M_c)_{\varphi_c}$ with $\varphi_c := \varphi\!\upharpoonright_{M_c}$ is also isomorphic to $L(\mathbb{F}_\infty)$. 
\end{theorem}

Remark again that faithful normal tracial states and periodic ones are particular cases of almost periodic ones by definition. The above theorem and Dykema's previous type II$_1$ results \cite{Dykema:DukeMathJ93} immediately imply the following corollary: 

\begin{corollary}\label{C4} The class of countable direct sums of hyperfinite von Neumann algebras with separable preduals, {\rm(}at most countably infinite{\rm)} amplifications of interpolated free group factors and full type III factors with separable preduals that admit discrete decomposition with discrete core $L(\mathbb{F}_\infty)\,\bar{\otimes}\,B(\ell^2)$ is closed under taking free products with respect to arbitrary almost periodic states. The resulting diffuse factor summand $M_c$ is either $L(\mathbb{F}_r)$ or $(L(\mathbb{F}_\infty)\,\bar{\otimes}\,B(\ell^2))\rtimes\mathrm{Sd}(M_c)$, or other words, the centralizer $(M_c)_{\varphi_c}$ is always $L(\mathbb{F}_r)$ {\rm(}with $r = \infty$ as long as $M_c$ is not finite{\rm)}.    
\end{corollary}  

It is unclear, at the present moment, whether or not the above class of von Neumann algebras is smallest among such ones including all the hyperfinite von Neumann algebras, because we do not know whether or not any type III factor with discrete core isomorphic to $L(\mathbb{F}_\infty)\,\bar{\otimes}\,B(\ell^2)$ can arise as a free product von Neumann algebra of hyperfinite ones. Perhaps we may replace full type III factors with separable preduals that admit discrete decomposition with discrete core $L(\mathbb{F}_\infty)\,\bar{\otimes}\,B(\ell^2)$ by almost periodic free Araki--Woods factors. Thus this issue is closely related to the question \cite[\S\S5.4]{Ueda:AdvMath11}. One possible and `ideal' way to answering this is probably an in-depth study of trace-scaling actions of countable discrete subgroups of $\mathbb{R}_+^\times$ on $L(\mathbb{F}_\infty)\,\bar{\otimes}\,B(\ell^2)$. Such a study must be hard and apparently needs completely new ideas, though several isomorphism  results on free Araki--Woods factors indirectly suggest its possibility.  

\medskip
The proof of Theorem \ref{T3} needs the next simple lemma. 

\begin{lemma}\label{L5} Let $N$ be a von Neumann algebra and $\alpha$ be a {\rm(}not necessarily faithful{\rm)} action on $N$ of a compact abelian group $G$ with Haar probability measure $dg$. Assume that the fixed-point algebra $N^\alpha$ is irreducible in $N$, that is, $(N^\alpha)'\cap N = \mathbb{C}$. Then the crossed product $N\rtimes_\alpha G$ is a {\rm(}possibly countably infinite{\rm)} direct sum of amplifications of $N^\alpha$. 
\end{lemma}
\begin{proof} By e.g.~\cite[Appendix]{Ueda:IMRN00} we have $\mathcal{Z}(N\rtimes_\alpha G) = \mathbb{C}1_N\,\bar{\otimes}\,C$ with a certain von Neumann subalgebra $C$ of $L(G)$. Since $L(G) \cong \ell^\infty(\widehat{G})$ is  atomic and abelian, so is $C$ too (see e.g.~\cite[Proposition 1.1]{Tomiyama:JFA72}). Choose each minimal (in $C$) projection $z \in C$. Then one can choose a minimal projection $e_\gamma = \int_G \overline{\langle g,\gamma\rangle}\lambda(g)\,dg \in \mathbb{C}\rtimes G = \mathbb{C}1_N\bar{\otimes}L(G)$ labelled by $\gamma \in \widehat{G}$ so that $e_\gamma \leq z$. One has
\begin{align*} 
e_\gamma x e_\gamma 
&= \int_{G\times G} \overline{\langle g_1 g_2,\gamma \rangle} \alpha_{g_1}(x)\lambda(g_1 g_2)\,dg_1 dg_2 \\
&= \int_G \left[\int_G \alpha_{g_1}(x)\,\overline{\langle g_2,\gamma\rangle} \lambda(g_2)\,dg_2\right]\,dg_1 = E_{N^\alpha}(x)e_\gamma
\end{align*} 
for every $x \in N \subseteq N\rtimes_\alpha G$, where $E_{N^\alpha}(-) = \int_G \alpha_g(-)\,dg$ is the canonical faithful normal conditional expectation from $N$ onto $N^\alpha$. Hence $e_\gamma (M\rtimes_\alpha G) e_\gamma = N^\alpha e_\gamma \cong N^\alpha$ is immediate. This argument is borrowed from \cite[Theorem 3.1]{Haga:TohokuMathJ76}. Hence $(N\rtimes_\alpha G)z$ is stably isomorphic to $e_\gamma (N\rtimes_\alpha G) e_\gamma \cong N^\alpha$, since $z$ is the central support of $e_\gamma$ in $N\rtimes_\alpha G$. 
\end{proof}

Here is the proof of Theorem \ref{T3}. 

\begin{proof} {\rm(Theorem \ref{T3})} Let us first prove that both $M_i\rtimes_{\sigma^{\varphi_i,\Gamma}}\Gamma$, $i=1,2$, fall in the class $\mathcal{R}_4$, whose definition was given in Proposition \ref{P2}. 

Firstly, assume that $(M_i,\varphi_i)$ is of the form (i). Since $G$ is abelian, $M_i\rtimes_{\sigma^{\varphi_i,\Gamma}}G$ must be hyperfinite by Connes's theorem (see e.g~\cite[Theorem XV.3.16]{Takesaki:Book}) so that it falls in the class $\mathcal{R}_4$, since we have known that it is semifinite, see Theorem \ref{T1}. 

Secondly, assume that $(M_i,\varphi_i)$ is of the form (ii). Then $\sigma_t^{\varphi}$ is inner for every $t \in \mathbb{R}$, and hence so is $\sigma_g^{\varphi_i,\Gamma}$ for every $g \in G$ since $M_i$ is full, that is, $\mathrm{Int}(M_i)$ is closed. Therefore, by (the proof of) \cite[Lemma 4.2]{Connes:JFA74} and \cite[Theorem X.1.7 (ii)]{Takesaki:Book} one has $M_i\rtimes_{\sigma^{\varphi_i,\Gamma}}G \cong M_i\,\bar{\otimes}\,L(G) \cong M_i\,\bar{\otimes}\,\ell^\infty(\Gamma)$, which falls in the class $\mathcal{R}_4$. 

Thirdly, assume that $(M_i,\varphi_i)$ is of the form (iii). By assumption we may and do assume that $M_i = (L(\mathbb{F}_\infty)\,\bar{\otimes}\,B(\ell^2))\rtimes_\rho\Lambda$ with $\Lambda = \mathrm{Sd}(M_i)$ and $\mathrm{Tr}\circ\rho_\lambda = \lambda^{-1}\mathrm{Tr}$ ($\lambda \in \Lambda$), where $\mathrm{Tr}$ is a faithful normal semifinite tracial weight on $L(\mathbb{F}_\infty)\,\bar{\otimes}\,B(\ell^2)$. We may and do also assume that $e := 1\otimes e_0 \in L(\mathbb{F}_\infty)\,\bar{\otimes}\,B(\ell^2)$ with minimal $e_0 \in B(\ell^2)^p$ takes $\mathrm{Tr}(e) =1$. Let $E : M_i \to L(\mathbb{F}_\infty)\,\bar{\otimes}\,B(\ell^2)$ be the canonical conditional expectation, and set $\psi := \mathrm{Tr}\circ E\!\upharpoonright_{eM_i e}$ with the above $e$. By \cite[Lemma X.1.13, Theorem X.1.17]{Takesaki:Book} one can easily see that $\psi$ is extremal $\Lambda$-almost periodic and $(eM_i e)_\psi = L(\mathbb{F}_\infty)\,\bar{\otimes}\,\mathbb{C}e_0 \cong L(\mathbb{F}_\infty)$. Since $eM_i e \cong M_i$, we get an extremal $\Lambda$-almost periodic state $\varphi_{i0}$ on $M_i$ whose centralizer $(M_i)_{\varphi_{i0}}$ is isomorphic to $L(\mathbb{F}_\infty)$. By definition $\Lambda$ must sit in $\Gamma$, and hence $\varphi_{i0}$ is extremal $\Gamma$-almost periodic too. By \cite[Lemma 4.2]{Connes:JFA74} and \cite[Theorem X.1.7 (ii)]{Takesaki:Book} we get $M_i\rtimes_{\sigma^{\varphi_i,\Gamma}}G \cong M_i\rtimes_{\sigma^{\varphi_{i0},\Gamma}}G$, which is a direct sum of amplifications of $(M_i)^{\sigma^{\varphi_{i0},\Gamma}} = (M_i)_{\varphi_{i0}} \cong L(\mathbb{F}_\infty)$ thanks to Lemma \ref{L5}. Therefore, $M_i\rtimes_{\sigma^{\varphi_i,\Gamma}}G$ falls in the class $\mathcal{R}_4$. 

Finally, assume that $(M_i,\varphi_i)$ is of the form (vi). Clearly $M_i\rtimes_{\sigma^{\varphi_i,\Gamma}}G$ is a direct sum of crossed products by $G$ considered above, since each $\sigma_g^{\varphi_i,\Gamma}$ fixes any element of $\mathcal{Z}(M_i)$. Consequently, $M_i\rtimes_{\sigma^{\varphi_i,\Gamma}}G$ falls in the class $\mathcal{R}_4$.  

By Theorem \ref{T1} Eq.\eqref{Eq5} together with the above discussion we observe that $M\rtimes_{\sigma^{\varphi,\Gamma}}G$ is a semifinite tracial amalgamated free product von Neumann algebra of two algebras from the class $\mathcal{R}_4$ with amalgamation over an atomic abelian von Neumann subalgebra, and hence Proposition \ref{P2} implies that $M\rtimes_{\sigma^{\varphi,\Gamma}}G$ falls again in the class $\mathcal{R}_4$ so that by Theorem \ref{T1} (1.3) the discrete core $\widehat{M_c}$ is stably isomorphic to either the unique hyperfinite type II$_1$ factor or an interpolated free group factor. However, we have already known that $(M_c)_{\varphi_c}$ has no central sequence as remarked in \S\S2.1, and hence by Eq.\eqref{Eq4} $\widehat{M_c}$ must be stably isomorphic to an interpolated free group factor $L(\mathbb{F}_r)$. By means of discrete decomposition the fundamental group of this $L(\mathbb{F}_r)$ must contain $\Gamma \neq \{1\}$, and thus $r=\infty$ by the famous dichotomy due to Dykema and R\u{a}dulescu, see \cite[Corollary 4.7]{Radulescu:InventMath94}. 
\end{proof} 

\subsection{General results on discrete cores associated with free products}\label{SS2.4} The purpose of this subsection is to clarify the present state of our knowledge on arbitrary free product von Neumann algebras. The consequences given below follow from several existing results e.g.~\cite{ChifanHoudayer:DukeMathJ10},\cite{HoudayerVaes:JMathPuresAppl1x},\cite{BoutonnetHoudayerRaum:Preprint12} (this list is not complete at all) in the deformation and rigidity theory with the help of our previous works \cite{Ueda:AdvMath11},\cite{Ueda:MRL11} and the present one. 

A recent work due to Boutonnet, Houdayer and Raum \cite{BoutonnetHoudayerRaum:Preprint12} and the results so far in the present paper altogether imply the next proposition. It is proved by the same argument of \cite[Theorem A]{BoutonnetHoudayerRaum:Preprint12} with replacing the continuous core, \cite[Theorem 5.1]{Ueda:PacificJMath99}, and \cite[Theorem 5.5]{HoudayerRicard:AdvMath11} by the discrete core, Theorem \ref{T1}, and Theorem \ref{T3} with Voiculescu's result \cite{Voiculescu:GAFA96}, respectively. We do give its proof for the sake of completeness. 

\begin{proposition}\label{P6} Let $M_i$, $i=1,2$, be non-trivial von Neumann algebras with separable preduals and $\varphi_i$, $i=1,2$, be almost periodic  states on them, respectively, such that at least one of them is not tracial. Then both the discrete core $\widehat{M}_c$ and the centralizer $(M_c)_{\varphi_c}$ with $\varphi_c := \varphi\!\upharpoonright_{M_c}$ of the resulting unique type III factor summand $M_c$ of the free product $(M,\varphi) = (M_1,\varphi_1)\star(M_2,\varphi_2)$ do never have any Cartan subalgebra.
\end{proposition}
\begin{proof} When both the $M_i$ are amenable (or equivalently hyperfinite), the desired assertion is immediate from  Voiculescu's result \cite{Voiculescu:GAFA96} with the help of Theorem \ref{T3}. Hence we may and do assume that $M_1$ is not amenable. Here is a standard claim. Its proof is left to the reader. 

\medskip\noindent
{\bf Claim:} There exists a unique non-zero projection $z \in \mathcal{Z}(M_1)$ such that $M_1 (1-z)$ is amenable but $M_1 e$ not for every non-zero projection $e \in \mathcal{Z}(M_1)$ with $e \leq z$.   

\medskip
As remarked at the end of \S\S2.1 $z \leq 1_{M_c}$, $M_c \cong zMz$ (since $M_c$ is of type III in this case), and $(zMz,\psi)$ with $\psi :=  \varphi(z)^{-1}\varphi\!\upharpoonright_{zMz}$ is the free product of $(M_1 z, \varphi_1(z)^{-1}\varphi_1\!\upharpoonright_{M_1 z})$ and non-trivial something. Since $z$ falls into $M_\varphi$, it is plain to see that $(zMz)_\psi = z((M_c)_{\varphi_c})z$. Hence, it is easy to confirm that $(M_c)_{\varphi_c}$ had a Cartan subalgebra if and only if so does $(zMz)_\psi = z((M_c)_{\varphi_c})z$. Since $z$ falls into $M_\varphi$ again, it is also easy to see that $\psi$ is almost periodic, and so are both the states in the free product description of $(zMz,\psi)$ by \cite[Corollary 3.1]{Ueda:MRL11}. Consequently, we may and do further assume that $M_1$ has no amenable summand. Note that $M = M_c$ and $\varphi = \varphi_c$ in the case. 

By Theorem \ref{T1} the discrete core $\widehat{M} = M\rtimes_{\sigma^{\varphi,\Gamma}}G$ is the semifinite amalgamated free product of $M_1\rtimes_{\sigma^{\varphi_1,\Gamma}}G$ and $M_2\rtimes_{\sigma^{\varphi_2,\Gamma}}G$ over $\mathbb{C}\rtimes G$. Since $M_1\rtimes_{\sigma^{\varphi_1,\Gamma}}G$ is not amenable due to Connes's result (see e.g.~\cite[Theorem XV.3.16]{Takesaki:Book}), we can apply the above claim to $M_1\rtimes_{\sigma^{\varphi_1,\Gamma}}G$ instead of $M_1$ and get a unique projection $\hat{z} \in \mathcal{Z}(M_1\rtimes_{\sigma^{\varphi_1,\Gamma}}G)$ such that $(M_1\rtimes_{\sigma^{\varphi_1,\Gamma}}G)(1-\hat{z})$ is amenable but $(M_1\rtimes_{\sigma^{\varphi_1,\Gamma}}G)e$ not for every non-zero projection $e \in \mathcal{Z}(M_1\rtimes_{\sigma^{\varphi_1,\Gamma}}G)$ with $e \leq \hat{z}$. Let $\theta^{(1)}_\gamma$, $\gamma \in \Gamma = \widehat{G}$, be the dual action of $\sigma^{\varphi_1,\Gamma}$. The uniqueness of $\hat{z}$ forces $\theta_\gamma^{(1)}(\hat{z}) = \hat{z}$ for all $\gamma \in \Gamma$, and thus $\hat{z}$ falls in the center of $(M_1\rtimes_{\sigma^{\varphi_1,\Gamma}}G)\rtimes_{\theta^{(1)}}\Gamma$. Moreover, $((M_1\rtimes_{\sigma^{\varphi_1,\Gamma}}G)\rtimes_{\theta^{(1)}}\Gamma)(1-\hat{z}) = ((M_1\rtimes_{\sigma^{\varphi_1,\Gamma}}G)(1-\hat{z}))\rtimes_{\theta^{(1)}}\Gamma$ becomes amenable due to Connes's result (see e.g.~\cite[Theorem XV.3.16]{Takesaki:Book}) again. By the Takesaki duality the unique $\hat{z}$ must be $1$, since $M_1$ (and hence $(M_1\rtimes_{\sigma^{\varphi_1,\Gamma}}G)\rtimes_{\theta^{(1)}}\Gamma \cong M_1\bar{\otimes}B(L^2(G))$) has no amenable summand. Consequently, we have shown that $M_1\rtimes_{\sigma^{\varphi_1,\Gamma}}G$ has no amenable summand too. By the construction of $\sigma^{\varphi_2,\Gamma}$ (see \cite[Proposition 1.1]{Connes:JFA74}) one has $\mathbb{C}\rtimes G = \mathbb{C}\,\bar{\otimes}\,L(G) \subseteq (M_2)_{\varphi_2}\,\bar{\otimes}\,L(G)$ sits inside $M_2\rtimes_{\sigma^{\varphi_2,\Gamma}}G$. Since $\varphi_2$ is almost periodic, one observes that $(M_2)_{\varphi_2} \neq \mathbb{C}$ due to e.g.~\cite[Lemma 2.1]{Ueda:MRL11}, and hence every non-zero projection $e = 1\otimes f \in \mathbb{C}\,\bar{\otimes}\,L(G) = \mathbb{C}\rtimes G$ satisfies that $e(M_2\rtimes_{\sigma^{\varphi_2,\Gamma}}G)e \supseteq (M_2)_{\varphi_2}\bar{\otimes}(L(G)f) \neq \mathbb{C}\,\bar{\otimes}\,(L(G)f) = (\mathbb{C}\rtimes G)e$. Consequently, \cite[Theorem 5.1]{BoutonnetHoudayerRaum:Preprint12} can be applied to the discrete core $\widehat{M}$. 

Suppose on the contrary that $M_\varphi$ has a Cartan subalgebra, say $A$, which must be diffuse since we have known that $M_\varphi$ is a type II$_1$ factor. Then $A\,\bar{\otimes}\,\ell^\infty$ becomes a Cartan subalgebra in $M_\varphi\,\bar{\otimes}\,B(\ell^2) \cong \widehat{M}$ by Eq.\eqref{Eq3}.  Thus the discrete core $\widehat{M}$ must have a Cartan subalgebra, say $D \cong A\bar{\otimes}\ell^\infty$. Since $\widehat{M}$ is a type II$_\infty$ factor, one can find a partial isometry $v \in \widehat{M}$ such that $v^* v \in D$, $vv^* \in \mathbb{C}\rtimes G$ and $\mathrm{Tr}(vv^*) < + \infty$ with the canonical trace $\mathrm{Tr}$ on $\widehat{M}$. Then $vDv^*$ becomes a Cartan subalgebra in $vv^*\widehat{M}vv^*$, and hence \cite[Theorem 5.1]{BoutonnetHoudayerRaum:Preprint12} shows that $vDv^* \preceq_{vv^*\widehat{M}vv^*}              (\mathbb{C}\rtimes G)vv^*$, a contradiction since $D$ is diffuse while $(\mathbb{C}\rtimes G)vv^*$ completely atomic. 
\end{proof} 

As pointed out in \cite[Corollary 4.3 (1)]{Ueda:AdvMath11} the diffuse factor summand $M_c$ is known to be always prime due to Chifan and Houdayer \cite[Theorem 5.2]{ChifanHoudayer:DukeMathJ10}; to be precise a supplementary argument due to Houdayer and Vaes \cite[Theorem 5.7]{HoudayerVaes:JMathPuresAppl1x} and a standard fact  e.g.~\cite[Lemma 5.2]{BoutonnetHoudayerRaum:Preprint12} are necessary. Here is one more proposition.  

\begin{proposition}\label{P7} Let $M_i$, $i=1,2$, be non-trivial von Neumann algebras with separable preduals and $\varphi_i$, $i=1,2$, be almost periodic states on them, respectively, such that at least one of them is not tracial. Then both the discrete core $\widehat{M}_c$ and the centralizer $(M_c)_{\varphi_c}$ with $\varphi_c := \varphi\!\upharpoonright_{M_c}$ of the resulting unique type III factor summand $M_c$ of the free product $(M,\varphi) = (M_1,\varphi_1)\star(M_2,\varphi_2)$ must be prime.
\end{proposition}
\begin{proof} Since we have known that $(M_c)_{\varphi_c}$ is a type II$_1$ factor and since any elements in the centers of both the $M_i$ falls into the centralizer of the free product state $\varphi$, the trick explained at the end of \S\S\ref{SS2.1} enables us to assume, as in the proof of Proposition \ref{P6} above, that $M = M_c$. It suffices to prove that the discrete core $\widehat{M}$, which must be a type II$_\infty$ factor, is prime. Theorem \ref{T1} says that $\widehat{M} = \widehat{M}_1 \star_D \widehat{M}_2$ equipped with a canonical faithful normal semifinite trace $\mathrm{Tr}$, where $\widehat{M}_i := M_i \rtimes_{\sigma^{\varphi_i,\Gamma}}G \supseteq D := \mathbb{C}\rtimes G$ with $\Gamma := \mathrm{Sd}(M)$ and $G := \widehat{\Gamma}$. We need to modify only the first two paragraphs of the proof of \cite[Theorem 5.2]{ChifanHoudayer:DukeMathJ10}. Suppose on the contrary that $\widehat{M} = N_1\,\bar{\otimes}\,N_2$ with two type II$_\infty$ factors $N_k$, $k=1,2$. By \cite[Theorem 4.1]{Ueda:AdvMath11} together with Eq.\eqref{Eq4} the discrete core $\widehat{M}$ is non-amenable so that we may and do assume that so is $N_1$. Choose finite projections $e_k \in N_k$, $k=1,2$, and set $p := e_1\otimes e_2 \in \widehat{M}$. Clearly one has $\mathrm{Tr}(p) < +\infty$. Set $P_1 := e_1 N_1 e_1\,\bar{\otimes}\,\mathbb{C}e_2$ and $P_2 := \mathbb{C}e_1\,\bar{\otimes}\,e_2 N_2 e_2$, both of which sit inside $p\widehat{M}p$. By \cite[Theorem 4.2]{ChifanHoudayer:DukeMathJ10} one has $P_2 = P_1' \cap p\widehat{M}p \preceq_{\widehat{M}} \widehat{M}_i$ for some $i \in \{1,2\}$. Namely, there exist $n \in \mathbb{N}$, a finite projection $q \in \widehat{M}_i^n$ (the $n$-amplification of $\widehat{M}_i$), a normal $*$-isomorphism $\rho : P_2 \to q\widehat{M}_i^n q$ and a partial isometry $v \in M_{1,n}(p\widehat{M})$ so that $v^* v \leq q$, $vv^* \leq p$ and $xv = v\rho(x)$ for $x \in P_2$. Set $Q := P_1 \vee \{vv^*\}$ inside $p\widehat{M}p$. Then $vv^* \in Q$ and $v^* v \in \rho(P_2)'\cap q\widehat{M}^n q$. Since $\rho(P_2)$ is a type II$_1$ factor, we see that $\rho(P_2) \not\preceq_{\widehat{M}} D^n$ so that $v^* v \in \rho(P_2)'\cap q\widehat{M}^n q \subseteq q\widehat{M}_i^n q$ by \cite[Theorem 2.4]{ChifanHoudayer:DukeMathJ10}. Thus we may and do assume that $v^* v = q$. Note that $v^* P_1 v \not\preceq_{\widehat{M}^n} D^n$ since $P_1$ is a type II$_1$ factor and that $v^* P_2 v = \rho(P_2)v^* v \subseteq q\widehat{M}_i^n q$. Hence by \cite[Theorem 2.4]{ChifanHoudayer:DukeMathJ10} again we get $v^* P_1 v \subseteq (v^* P_2 v)'\cap q\widehat{M}^n q \subseteq q \widehat{M}_i^n q$. Consequently, $v^* Q v = v^* P_1 v \vee \{v^* v\}'' \subseteq q\widehat{M}_i^n q$. Hence, with these $Q$ and $v$ we can follow the third and fourth paragraphs of the proof of \cite[Theorem 5.2]{ChifanHoudayer:DukeMathJ10} (this part is easier to handle than the original argument since our $P_2$ is finite) and then can do the proof of \cite[Theorem 5.7]{HoudayerVaes:JMathPuresAppl1x} without any change; thus we get a non-trivial projection $r \in D$ so that $r\widehat{M}_{i'}r = Dr$ with $i' \neq i$. This is impossible since $(M_{i'})_{\varphi_{i'}}$ is non-trivial thanks to e.g.~\cite[Lemma 2.1]{Ueda:MRL11} (see the proof of Proposition \ref{P6} above).      
\end{proof} 

The same pattern of the above argument (with the help of a standard fact, see e.g.~\cite[Lemma 5.2]{BoutonnetHoudayerRaum:Preprint12}) shows the next remark. The details are left to the reader, since the main targets of the present paper are discrete cores.

\begin{remark}\label{R8} Let $M_i$, $i=1,2$, be non-trivial von Neumann algebras with separable preduals and $\varphi_i$, $i=1,2$, be faithful normal states on them, respectively. If the resulting unique diffuse summand $M_c$ of the free product $(M,\varphi) = (M_1,\varphi_1)\star(M_2,\varphi_2)$ is a type III$_1$ factor, then the continuous core $\widetilde{M}_c = M_c \rtimes_{\sigma^{\varphi_c}}\mathbb{R}$ with $\varphi_c := \varphi\!\upharpoonright_{M_c}$ is a prime type II$_\infty$ factor.
\end{remark} 

The above two assertions tell us that the diffuse factor summand $M_c$ and all its possible `canonical' factors must be prime. Hence the question of primeness of free product von Neumann algebras has been solved completely. Remark that Gao and Junge \cite{GaoJunge:IMRN07} proved, based on Ozawa's method \cite{Ozawa:ActaMath04}, that if both $M_i$, $i=1,2$, are hyperfinite, then $M$ must be `state solid', a stronger property than primeness. However this property cannot hold in general. 

In contrast to the question of primeness, it seems still unsolved whether or not the continuous core $\widetilde{M}_c$ of the diffuse factor summand $M_c$ can possess a Cartan subalgebra when $M_c$ is of type III$_1$ (and $\varphi_c$ not almost periodic). This question seems untouched even when both the $M_i$ are hyperfinite. The same question for free Araki--Woods factors was already settled by Houdayer and Ricard \cite[Theorem D (1)]{HoudayerRicard:AdvMath11}.  

\subsection{Second precise structural results}\label{SS2.5}  
The case when one of the $\varphi_i$, $i=1,2$, is tracial is important in actual applications of type III free product factors. In fact, any free Araki--Woods factor with admitting discrete decomposition is of such form (see \cite{Shlyakhtenko:PacificJMath97}) and also any countable discrete subgroup of $\mathbb{R}_+^\times$ is obtainable as the fundamental group of the centralizer of such a free product factor (see \cite{Houdayer:Crelle09}). Here we prove precise descriptions of the centralizers of such free product states by combining Theorem \ref{T1} and Dykema's devices \cite{Dykema:MathProcCambPhilSoc04} (see \S3). 

\begin{proposition}\label{P9} Let $M_i$, $i=1,2$, be non-trivial von Neumann algebras with separable preduals and $\varphi_i$, $i=1,2$, be faithful normal states on them, respectively. 

{\rm(9.1)} Assume that $M_1$ is a type III factor, $\varphi_1$ almost periodic such that its centralizer $(M_1)_{\varphi_1}$ is a factor {\rm(}i.e., $\varphi_1$ is extremal{\rm)}, $M_2$ either a diffuse hyperfinite von Neumann algebra or a type II$_1$ factor, and finally $\varphi_2$ tracial. Then $M$ is a type III factor that admits discrete decomposition, and the centralizer $M_\varphi$ of the free product state $\varphi$ is isomorphic to 
\begin{equation*} 
\begin{cases}
(M_1)_{\varphi_1}\star L(\mathbb{F}_\infty)  &(\text{when $M_2$ is a diffuse hyperfinite von Neumann algebra}), \\
(M_1)_{\varphi_1}\star\Big(\bigstar_{\gamma\in\Gamma} (M_2)^\gamma\Big) 
&(\text{when $M_2$ is a type II$_1$ factor})
\end{cases}
\end{equation*}
with the Sd-invariant $\Gamma := \mathrm{Sd}(M)$ of $M$, where $(M_2)^\gamma$ denotes the $\gamma$-amplification of the type II$_1$ factor $M_2$. 

{\rm(9.2)} Assume that $M_1$ is of atomic type I, $\varphi_1$ not tracial {\rm(}automatically being almost periodic{\rm)}, and $M_2$ a type II$_1$ factor with tracial $\varphi_2$. Then $M$ is a type III factor that admits discrete decomposition, and the centralizer $M_\varphi$ of the free product state $\varphi$ is isomorphic to $\bigstar_{\gamma\in\Gamma} (M_2)^\gamma$, where $\Gamma$ and $(M_2)^\gamma$ are as in {\rm(1)}. 
\end{proposition}
\begin{proof} 
(9.1) By \cite[Theorem 4.1]{Ueda:AdvMath11} and Theorem \ref{T1} $M\rtimes_{\sigma^{\varphi,\Gamma}}G$ is a type II$_\infty$ factor and also a semifinite tracial amalgamated free product of $M_1\rtimes_{\sigma^{\varphi_1,\Gamma}}G$ and $M_2\rtimes_{\sigma^{\varphi_2,\Gamma}}G$ with amalgamation over the atomic abelian von Neumann subalgebra $\mathbb{C}\rtimes G \cong \ell^\infty(\Gamma)$. In this case $M_2\rtimes_{\sigma^{\varphi_2,\Gamma}}G$ is nothing but a direct sum of infinitely many copies of $M_2$. Moreover, by \cite[Lemma 4.8]{Connes:JFA74} the point spectrum of the modular operator $\Delta_{\varphi_1}$ must be $\Gamma$ (due to the assumption that $(M_1)_{\varphi_1}$ is a factor). Remark also that $e_1 (M_1\rtimes_{\sigma^{\varphi_1,\Gamma}}G)e_1 \cong (M_1)_{\varphi_1}$ due to the proof of Lemma \ref{L5}. Since $\sum_{\gamma\in\Gamma}\mathrm{Tr}(e_\gamma) = \sum_{\gamma\in\Gamma}\gamma^{-1} = +\infty$ by Theorem \ref{T1} (1.1), the desired assertion immediately follows from \cite[Corollary 2.2]{Dykema:MathProcCambPhilSoc04} and its proof (strictly speaking it shows only the latter case, but the proof perfectly works for the former too). 

\medskip
(9.2) As in the proof of Theorem \ref{T3} one can see that $M_1\rtimes_{\sigma^{\varphi_1,\Gamma}}G \cong M_1\,\bar{\otimes}\,L(G) \cong M_1\,\bar{\otimes}\,\ell^\infty(\Gamma)$, being of atomic type I too. Moreover, Theorem \ref{T1} and e.g.~\cite[Theorem 4.3 (1)]{Ueda:JLMS13} show that $\mathcal{Z}(M_1\rtimes_{\sigma^{\varphi_1,\Gamma}}G) \cap (\mathbb{C}\rtimes G) = \mathbb{C}$, since $M\rtimes_{\sigma^{\varphi,\Gamma}}G$ has been known to be a factor. Hence \cite[Corollary 3.2 (B)]{Dykema:MathProcCambPhilSoc04} can be applied to $M\rtimes_{\sigma^{\varphi,\Gamma}}G$. Thus it remains only to confirm one of the conditions (i)--(iii) there. 

To do so one needs to examine $M_1\rtimes_{\sigma^{\varphi_1,\Gamma}}G \supseteq \mathbb{C}\rtimes G$ and the $\mathrm{Tr}$ constructed in Theorem \ref{T1} (1.1). Write $M_1 = \sum_{k\geq1}^\oplus B(\mathcal{H}_k)$, and for each $k\geq1$ one can choose a complete set of minimal projections $f_{k1},f_{k2},\dots \in B(\mathcal{H}_k) \subset M_1$ and a decreasing sequence $\gamma_{k1}=1 \geq \gamma_{k2} \geq \cdots >0$ in $\Gamma$ so that the density operator of $\varphi_1\!\upharpoonright_{B(\mathcal{H}_k)}$ is diagonalized as  $c_k\sum_{j\geq1} \gamma_{kj} f_{kj}$. Then the bijective $*$-homomorphism $M_1\rtimes_{\sigma^{\varphi_1,\Gamma}}G \cong M_1\,\bar{\otimes}\,L(G)$ is given by sending $x \in M_1 \subset M_1\rtimes_{\sigma^{\varphi_1,\Gamma}}G$ and $\lambda(g)$ to $x\otimes1$ and $u(g)\otimes\lambda_g$ with $u(g) = \sum_{k\geq1}\sum_{j\geq1} \langle \gamma_{kj},g\rangle f_{kj}$, where the $\lambda_g$, $g \in G$, denote the (left) regular representation of $G$ on $L^2(G)$. Accordingly, via the isomorphism $L(G) \cong \ell^\infty(\Gamma)$ every `canonical' projection $e_\gamma \in \mathbb{C}\rtimes G \subset M_1\rtimes_{\sigma^{\varphi_1,\Gamma}}G$, $\gamma \in \Gamma$, is transformed to $p_\gamma := \sum_{k\geq1}\sum_{j\geq1} f_{kj}\otimes \delta_{\gamma_{kj}^{-1}\gamma} \in M_1\,\bar{\otimes}\,\ell^\infty(\Gamma)$, and the $\mathrm{Tr}$ to the trace $\mathrm{Tr}_0$ determined by 
$$
\mathrm{Tr}_0(f_{kj} \otimes \delta_\gamma) = \mathrm{Tr}_0(f_{kj}\otimes1\cdot p_{\gamma_{kj}\gamma})\,(= \mathrm{Tr}(E_{\varphi_1}(f_{kj}) e_{\gamma_{kj}\gamma}) = c_k\gamma_{kj}\mathrm{Tr}(e_{\gamma_{kj}\gamma})\,)= c_k/\gamma
$$ 
for every $(k,j)$ and every $\gamma\in\Gamma$, where the $\delta_\gamma$, $\gamma \in \Gamma$, are the canonical minimal projections of $\ell^\infty(\Gamma)$. 

The rest of the proof heavily depends on the discussion in \S3, a modification of the original proof of \cite[Proposition 3.1]{Dykema:MathProcCambPhilSoc04}; hence we advise the reader to consult \S3 before going to the following arguments. One can choose an appropriate pair $k_0,j_0$ in such a way that $\gamma_* := \gamma_{k_0 j_0} \in (0,1)$. Let $\Gamma = \bigcup_{n\in\mathbb{N}_0} I_n$ be as in Lemma \ref{L11} in \S3. Choosing a suitable total ordering on each $I_n$ we may and do assume that $\gamma_*^{m_1} \prec \gamma_*^{m_2}$ whenever $m_1 < m_2$ and $\gamma_*^{m_1}, \gamma_*^{m_2} \in I_n$ for the same $n$, where `$\prec$' denotes the ordering on $\Gamma$ that we are employing. Here is a claim. 

\medskip\noindent
{\bf Claim:} There is a strictly increasing sequence $\{m_k\}$ of natural numbers such that $\gamma_*^{m_k -1} \prec \gamma_*^{m_k}$ holds for every $k = 1,2,\dots$. 

\noindent($\because$) Set $m_1 := 1$, and then $\gamma_*^{m_1 - 1} = 1 \in I_0$ and $\gamma_*^{m_1} = \gamma_* = \gamma_{k_0 j_0} \in I_1$; hence $\gamma_*^{m_1 - 1} \prec \gamma_*^{m_1}$ holds. Assume that we have already chosen $m_1 < \cdots < m_k$ with the desired property. The next $m_{k+1}$ can be chosen as follows. Let $n$ be the unique natural number so that $\gamma_*^{m_k} \in I_n$. Since our ordering is total, only two possibilities exists; namely either $\gamma_*^{m_k} \prec \gamma_*^{m_k + 1}$ or $\gamma_*^{m_k + 1} \prec \gamma_*^{m_k}$. It suffices to define $m_{k+1} := m_k + 1$ in the former case. Thus we consider the latter case. Let $n'$ be the unique natural number such that $\gamma_*^{m_k + 1} \in I_{n'}$. Our ordering forces $n'$ to be strictly smaller than $n$. There are only two possibilities; the unique $n''$ with $\gamma_*^{m_k + 2} \in I_{n''}$ satisfies either $n' \leq n''$ or $n'' < n'$. In the former case $\gamma_*^{m_k + 1} \prec \gamma_*^{m_k + 2}$ holds; hence it suffices to define $m_{k+1} := m_k + 2$. In the latter case we repeat the same argument for $\gamma_*^{m_k + 3}$, and this procedure will certainly stop after finitely many steps since $n$ is finite. Consequently, one can choose $l \geq 1$ in such a way that $\gamma_*^{m_k + l -1} \prec \gamma_*^{m_k + l}$, and the desired $m_{k+1}$ can be chosen to be $m_k + l$. Hence the claim is proved by induction. \qed

\medskip 
Notice that $m_k \leq m_{k+1} - 1 \lneqq m_{k+1}$ and $\gamma_*^{m_k - 1} \prec \gamma_*^{m_k}$ for every $k$. Moreover, we have $f_{k_0 1}\otimes\delta_{\gamma_*^{m_k -1}} \leq p_{\gamma_*^{m_k -1}}$, $f_{k_0 j_0}\otimes\delta_{\gamma_*^{m_k -1}} \leq p_{\gamma_*^{m_k}}$ and $\mathrm{Tr}_0(f_{k_0 1}\otimes\delta_{\gamma_*^{m_k -1}}) = c_{k_0}/\gamma_*^{m_k -1} \to +\infty$ as $k\to\infty$. We can choose $\mathrm{Tr}_0(f_{k_0 1}\otimes \delta_{\gamma_*^{m_k -1}})$ as `$\gamma(i)$' with $i=\gamma_*^{m_k}$ and $I=\Gamma$ in \S3. Hence the condition (ii) of \cite[Corollary 3.2 (B)]{Dykema:MathProcCambPhilSoc04} trivially holds under these choices of ordering and minimal subprojections. Here the reader should notice the comment in the last sentence of \S3.  
\end{proof} 

The next simple application of the above proposition seems interesting. 

\begin{remark}\label{R10}{\rm 
Let $\Gamma$ be an arbitrary countable subgroup of $\mathbb{R}_+^*$, and enumerate $\{\gamma_1,\gamma_2,\dots\}=\Gamma\cap(0,1)$. Consider $M_1 := \sum_{k\geq1}^\oplus M_2(\mathbb{C})$ and the state $\varphi_1$ is given by the density matrix: 
$$
\sideset{}{^\oplus}\sum_{k\geq1} \frac{1}{2^k(1+\gamma_k)}(e_{11}^{(k)}+\gamma_k e_{22}^{(k)}) 
$$
with standard matrix units $e_{ij}^{(k)}$ in $M_2(\mathbb{C})$. The pair $(M_1,\varphi_1)$ fulfills the requirement of the above assertion (2), and thus the centralizer of the resulting free product state is $\bigstar_{\gamma\in\Gamma}(M_2)^\gamma$. This relates \cite[Corollary 6.5 (4)]{IoanaPetersonPopa:Acta08} to \cite[Theorem 5.10]{Houdayer:Crelle09}.}
\end{remark} 

\section{Addition to \cite{Dykema:MathProcCambPhilSoc04}} 

We used \cite[Corollary 3.2]{Dykema:MathProcCambPhilSoc04}, a corollary of \cite[Proposition 3.1]{Dykema:MathProcCambPhilSoc04}, in the proof of Proposition \ref{P6} (2) above, but \cite[Proposition 3.1]{Dykema:MathProcCambPhilSoc04} needs the following ($\diamondsuit$): `For every $m \in I\setminus\{1\}$, there exists a minimal projection $e_m$ in $A$ such that $e_m \leq p_1 + \cdots + p_{m-1}$ and $e_m$ is equivalent to a subprojection of $p_m$' with the notations there. Such $e_m$ does not exist for an arbitrary enumeration (or ordering by natural numbers) on $\{p_i\}_{i\in I}$, and it seems not so easy in general to choose a special enumeration on $\{p_i\}_{i\in I}$ for which ($\diamondsuit$) holds, when $I$ is infinite. Since some part of our proof of Proposition \ref{P9} depends on the careful choice of such $e_m$, we will give a way of avoiding ($\diamondsuit$) for the reader's convenience, though the potential reader can probably modify the original proof. We follow the notations in \cite[\S3]{Dykema:MathProcCambPhilSoc04} in principle.  

\begin{lemma}\label{L11} Fix a distinguished element $o \in I$. Define $I_0 := \{o\}$, $I_1 := \{ i \in I\setminus I_0\,|\,p(0)Ap_i \neq \{0\}\}$ with $p(0) := p_o$ and inductively $I_{n+1} := \{ i \in I \setminus \bigcup_{s=0}^n I_s\,|\, p(n)Ap_i \neq \{0\}\}$ with $p(n) := \sum_{s=0}^n \sum_{i \in I_s} p_i$. Then $I = \bigcup_{n \in \mathbb{N}_0} I_n$, a disjoint union, with $\mathbb{N}_0 := \{0\}\sqcup\mathbb{N}$. 
\end{lemma}
\begin{proof} Set $I_\infty := \bigcup_{n \in \mathbb{N}_0} I_n$, $p := \sum_{i\in I_\infty} p_i$ and $q := \sum_{i \in I\setminus I_\infty} p_i$. Clearly $p \neq 0$, and by the construction of the $I_n$'s and $p(n)$'s we have $p(n) A p_i = \{0\}$ for every $i \in I\setminus I_\infty$. Since $p(n) \nearrow p$ as $n\to\infty$, we conclude that $pAp_i = \{0\}$ for every $i \in I\setminus I_\infty$. Similarly we have $pAq = \{0\}$, implying $c_p^A c_q^A = 0$. Since $p+q = 1$, $p \leq c_A(p)$, $q \leq c_A(q)$ and $c_A(p) c_A(q) = 0$, we observe that $p = c_A(p), q = c_A(q) \in B \cap \mathcal{Z}(A) = \mathbb{C}$, and hence $q = 0$, implying $I = I_\infty$.  
\end{proof}  

The decomposition $I = \bigcup_{n\in\mathbb{N}_0} I_n$ is used, and the new well-ordering on $I$ is made in such a way that every $I_n$ is $\{(n,1),(n,2),\dots\}$ and that $(n_1,m_1) \prec (n_2,m_2)$ is defined by $n_1 < n_2$, or $n_1=n_2$ and $m_1 < m_2$. (Note that if all $I_n$'s are finite sets, our ordering can become a desired choice of enumeration on $I$ with ($\diamondsuit$), and thus the proof of \cite[Proposition 3.1]{Dykema:MathProcCambPhilSoc04} works just as it is.) With this new ordering and replacing $p_1$ by $p_{(0,1)} = p(0) = p_o$ the same argument as in \cite[Proposition 3.1]{Dykema:MathProcCambPhilSoc04} still essentially works with reading the `$m$' and `$m+1$' there as $(n,m), (n,m+1) \in I_n$, respectively, since the same assertion ($\diamondsuit$) holds for our new ordering. To be more precise the following slight modifications are necessary: For each $n=1,2,\dots$ we consider the new element $(n,0)$ strictly smaller than every element of $I_n$ and strictly larger than every element of $\bigcup_{s \leq n-1} I_s$, and set $\mathcal{P}((n,0)) := p(n-1) Q \vee p(n-1) A p(n-1)$ ({\it n.b.}~$p(n-1) := \sum_{s=0}^{n-1}\sum_{i \in I_s} p_i$ is nothing but $\sum_{i \prec (n,1)} p_i$). With $o = (0,1)$ (playing the r\^{o}le of `$1$' there) and identifying $m = (n,m)$, $m+1=(n,m+1)$ we follow the proof of \cite[Proposition 3.1]{Dykema:MathProcCambPhilSoc04} for a while. For each $k \in K_{(n,m)}^{(1)} = \{k(0),\dots,k(l),\dots\}$ (which is not empty thanks to our new ordering) one can find $\langle k \rangle \in I$ with $\langle k\rangle \prec (n,m+1)$ in such a way that a minimal (in $A$) subprojection, say $e_k$, of $p_{\langle k \rangle}$ is equivalent (in $A$) to a subprojection of $p_{(n,m+1)}$. By the argument there, we get \cite[Eq.(11)]{Dykema:MathProcCambPhilSoc04} with replacing $p_{[1,m]}$, $\mathcal{P}(m)$ and $\beta[1,m]$ there by $p_{\{o,\langle k(0)\rangle\}} := p_o + p_{\langle k(0)\rangle}$, $p_{\{o,\langle k(0)\rangle\}}\mathcal{P}(n,m)p_{\{o,\langle k(0)\rangle\}}$ and $\beta(o)+\beta(\langle k(0)\rangle)$, respectively, and hence
\begin{align*} 
p_o\mathcal{R}(0)p_o &\overset{\sim}{\longrightarrow} 
p_o\mathcal{P}(n,m) p_o \star \Big[\frac{\gamma(n,m+1)}{\beta(o)},\mathcal{N}(n,m+1)_{\frac{\gamma(n,m+1)}{\beta(n,m+1)}}\Big] 
\\
&\overset{\sim}{\longrightarrow} 
p_o\mathcal{P}(n,m) p_o \star \Big[\frac{\gamma(n,m+1)}{\beta(o)},Q(n,m+1)_{\frac{\gamma(n,m+1)}{\beta(n,m+1)}}\Big]\star L\big(\mathbb{F}_{(\frac{\beta(n,m+1)}{\beta(o)})^2 s(n,m+1)}\big)
\end{align*} 
with $\gamma(m+1) := \mathrm{Tr}_A(e_{k(0)})$ and $s(n,m+1) := \mathrm{fdim}(p_{(n,m+1)}Ap_{(n,m+1)})$ (Dykema's free dimension or the free entropy dimension which is well-defined in this case) in the normalized trace obtained from $\mathrm{Tr}_A$, where \cite[Lemma 4.7, Proposition 4.8, Theorem 3.9]{Dykema:JFA02} are used. Similarly, one has \cite[Eq.(12)]{Dykema:MathProcCambPhilSoc04} with replacing $p_{[1,m]}$ and $\beta[1,m]$ by $p_o + p_{\langle k(l)\rangle}$ and $\beta(o)+\beta(\langle k(l)\rangle)$, respectively, and thus
\begin{align*} 
p_o\mathcal{R}(l)p_o \overset{\sim}{\longrightarrow} 
p_o\mathcal{R}(l-1) p_o \star L\big(\mathbb{F}_{\frac{\alpha(k(l))}{\beta(o))^2}}\big).  
\end{align*} 
These two isomorphisms immediately give \cite[Eq.(14)]{Dykema:MathProcCambPhilSoc04} with our notations. In this way, one can avoid the use of `$\beta[o,m]$' which may become $+\infty$ in our ordering. By taking the inductive limit as $m\to\infty$ on $I_n$ we get, with our notations,  
\begin{align*}
p_o\mathcal{P}((n+1,0))p_o \overset{\sim}{\longrightarrow} 
p_o\mathcal{P}((n,0)) p_o\star L(\mathbb{F}_{r_n}) \star \Big(\bigstar_{i\in I_n} \Big[\frac{\gamma(i)}{\beta(o)},Q(i)_{\frac{\gamma(i)}{\beta(i)}}\Big]\Big)
\end{align*}
with $r_n := \frac{1}{\beta(o)^2}\sum_{i \in I_n}(\beta(i)^2 - \gamma(i)^2 - \sum_{j \in J_i} \alpha(j)^2)$, where $J_o := \{ j \in J\,|\, q_j p_o \neq 0\}$, $J_i := \{ j \in J\,|\, q_j p_i \neq 0,\, q_j p_{i'} = 0\,(i' \prec i)\}$ ($i \in I\setminus\{o\}$), and $\alpha(j) := \mathrm{Tr}_A(f)$ with a (any) minimal projection $f \in Aq_j$. Since the $\mathcal{P}((n,0))$, $n=1,2,\dots$, generate the whole $\mathcal{P}$ as a von Neumann algebra, taking the inductive limit as $n\to\infty$ again one gets the desired formula:  
\begin{equation*}
p_o \mathcal{P}p_o \cong 
Q(o) \star L(\mathbb{F}_r) \star \Big(\bigstar_{i \in I\setminus\{o\}}\Big[\frac{\gamma(i)}{\beta(o)},Q(i)_{\frac{\gamma(i)}{\beta(i)}}\Big]\Big)
\end{equation*}
with $r = \frac{1}{\beta(o)^2}\big((\beta(o)^2 - \sum_{j \in J_o}\alpha(j)^2) + \sum_{i\in I\setminus\{o\}}(\beta(i)^2 - \gamma(i)^2 - \sum_{j \in J_i} \alpha(j)^2)\big)$. In this way, the argument of \cite[Proposition 3.1]{Dykema:MathProcCambPhilSoc04} still essentially works without ($\diamondsuit$). Finally, we remark that {\it $\gamma(i)$ can be chosen to be $\mathrm{Tr}_A(e)$ of any minimal subprojection $e$ of some $p_{i'}$ with $i' \prec i$, which is equivalent {\rm(}in $A${\rm)} to a subprojection of $p_i$}. 

\section{A simplified proof of Proposition \ref{P2}}

We will prove Proposition \ref{P2} in a simplified way, by utilizing only \cite{Dykema:DukeMathJ93} and \cite{Dykema:AmerJMath95}. Our main concerns here are delicate type II$_\infty$ problems, since the treatments on such problems in \cite{Redelmeier:arXiv:1207.1117} are sketchy. Our proof is indeed simplified but itself does not work for re-proving the full assertions in \cite{Dykema:AmerJMath95},\cite{Dykema:BLMS11},\cite{DykemaRedelmeier:arXiv11}. Assume that the von Neumann algebras in question have separable preduals. 

\subsection{Hyperfiniteness along atomic abelian subalgebras}
 
Let $M$ be a hyperfinite, semifinite von Neumann algebra equipped with a faithful normal (semifinite) trace $\mathrm{tr}$, and $D = \sum_{i\geq 1}^\oplus \mathbb{C}p_i$ be its unital, commutative, atomic von Neumann subalgebra. Let $\mathrm{ctr}$ be a generalized center-valued trace so that $\mathrm{tr} = \varphi\circ\mathrm{ctr}$ with a fixed faithful normal state $\varphi$ on $\mathcal{Z}(M)$ (see \cite[p.329--332]{Takesaki:Book},\cite[Theorem 2.7 and its remark]{Haagerup-JFA79-1}). Assume that all the $\mathrm{tr}(p_i)$ are finite. Although the next proposition looks easy to prove, its proof is unexpectedly involved. The proof below seems new, and giving it is one of the main aims of this section. 

\begin{proposition}\label{P-a1} There is a sequence of finite dimensional $*$-subalgebras $A_1 \subseteq A_2 \subseteq A_3 \subseteq \cdots$ of $M$ with the following properties{\rm:}
\begin{itemize} 
\item $q_n : = \sum_{i \leq n} p_i \nearrow 1 = 1_M$ as $n\to\infty$ and $q_n = 1_{A_n}$. 
\item $q_n \in D$ and $Dq_n \subseteq A_n$ for every $n \geq 1$. 
\item all the $A_n$ generate the whole $M$ as von Neumann algebra. 
\end{itemize} 
If $M$ is of type II, then the first subalgebras $A_1$ can be chosen to be non-trivial. 
\end{proposition} 

Up until the end of the proof of Lemma \ref{L-a4} we assume that $M$ is of type II. Remark that \emph{$\mathrm{ctr}(p_i) \in \mathcal{Z}(M)$ does not hold in general}, and hence we need the next lemma. 

\begin{lemma}\label{L-a2} For every $i \in I$ there exists a sequence of projections $p_{i,l} \in M$ so that $p_i = \sum_{l\geq 1} p_{i,l}$ and $\mathrm{ctr}(p_{i,l}) \leq 1$ for every $l \geq 1$. 
\end{lemma}
\begin{proof} Let $p \in \mathcal{Z}(M)^p$ satisfy $p \leq p_i$. In view of that $\mathrm{ctr}$ is an operator-valued weight it follows from $\mathrm{tr}(p_i) < +\infty$ that $\mathrm{ctr}(p)$ must be a (possibly unbounded) positive self-adjoint operator affiliated with $\mathcal{Z}(M)$. Hence we have the spectral decomposition $\mathrm{ctr}(p) = \int_0^{+\infty}\lambda\,e(d\lambda)$. Then $q := pe \neq 0$ for some $e := e([0,2m]) \in \mathcal{Z}(M)^p$ with $m > 0$. By \cite[Proposition V.1.35]{Takesaki:Book} there exist $2m$ equivalent projections $f_1,\dots,f_{2m}$ whose sum is $q$, and hence $f_1 \leq p$ and $\mathrm{ctr}(f_1) \leq e \leq 1$ hold. Choose a maximal orthogonal family $\{ e_\lambda\,|\,\lambda \in \Lambda\}$ of projections in $M$ such that $e_\lambda \leq p_i$ and $\mathrm{ctr}(e_\lambda) \leq 1$ hold for every $\lambda$. If $p := p_i - \sum_\lambda e_\lambda \neq 0$, then what we proved above causes a contradiction to the maximality. Hence we have $p_i = \sum_\lambda e_\lambda$, and the desired assertion follows since our $M$ has separable predual (hence it is $\sigma$-finite).     
\end{proof} 

\begin{lemma}\label{L-a3} For every $i,l \geq 1$ there exist two sequences of projections $p_{i,l,k} \in M $, $k \geq 1$, and $c_{i,l,k} \in \mathcal{Z}(M)$, $k \geq 1$, so that $p_{i,l} = \sum_{k\geq 1} p_{i,l,k}$ and $\mathrm{ctr}(p_{i,l,k}) = 2^{-k} c_{i,l,k}$ for every $k \geq 1$. 
\end{lemma} 
\begin{proof} Since $0 \leq \mathrm{ctr}(p_{i,l}) \leq 1$, its spectral decomposition gives $\mathrm{ctr}(p_{i,l}) = \sum_{k\geq1} 2^{-k} c_{i,l,k}$ with $c_{i,l,k} \in \mathcal{Z}(M)^p$. Let $s$ be the support projection of $\mathrm{ctr}(p_{i,l})$, and decompose $s = \sum_{m\geq 2} s_m$ in $\mathcal{Z}(M)^p$ so that $\mathrm{ctr}(p_{i,l}s_m) = \mathrm{ctr}(p_{i,l})s_m \geq m^{-1}s_m$ for every $m \geq 2$. For each $m$ there exists a positive $z_m \in \mathcal{Z}(M)$ so that $z_m\mathrm{ctr}(p_i s_m) = s_m$. Consider the normalized center-valued trace $\tau_m : x \in p_{i,l}s_m M p_{i,l}s_m \mapsto \mathrm{ctr}(x)z_m p_i s_m$. Using \cite[Corollary 3.14]{Kadison:AmerJMath84} inductively one can find an orthogonal sequence of projections $e_k^{(m)} \in p_{i,l}s_m M p_{i,l}s_m$, $k \geq 1$, so that $\tau_m(e_k^{(m)}) = 2^{-k} c_{i,l,k} z_m p_{i,l} s_m$ for every $k \geq 1$. Then, $\mathrm{ctr}(e_k^{(m)})p_{i,l} = \mathrm{ctr}(p_{i,l}s_m)\tau_m(e_k^{(m)}) = 2^{-k} c_{i,l,k}s_m p_{i,l}$ so that $\mathrm{ctr}(e_k^{(m)}) = 2^{-k} c_{i,l,k}s_m$. Letting $p_{i,l,k} := \sum_{m \geq 2} e_k^{(m)}$ one has $p_{i,l,k} \leq p_{i,l}$, $p_{i,l,k}p_{i,l,k'} = 0$ if $k\neq k'$, and $\mathrm{ctr}(p_{i,l,k}) = 2^{-k} c_{i,l,k}$. One has $\mathrm{ctr}(p_{i,l} -\sum_{k\geq1}p_{i,l,k}) = \mathrm{ctr}(p_{i,l}) - \sum_{k\geq1} \mathrm{ctr}(p_{i,l,k}) = \mathrm{ctr}(p_{i,l}) - \sum_{k\geq1} 2^{-k} c_{i,l,k} = 0$, and hence $p_{i,l} = \sum_{k\geq1} p_{i,l,k}$. 
\end{proof}

Remark that $c_M(p_{i,l,k}) = c_{i,l,k}$; this can easily be checked. 

\begin{lemma}\label{L-a4} Let $B_0$ be a finite dimensional $*$-subalgebra of $M$ with a matrix unit system $\{e_0^{(s)}(i,j)\}_{(i,j),s}$,  $e_1,\dots,e_n$ be an orthogonal, finite family of projections in $M$ and $x_1,\dots,x_{n'}$ be a finite family of elements in $M$ such that 
\begin{itemize} 
\item[(a)] $\mathrm{ctr}(e_0^{(s)}(i,j)) = \delta_{ij}\,2^{-k_0^{(s)}} c_0^{(s)}$ with some $k_0^{(s)} \in \mathbb{N}$ and $c_0^{(s)} \in \mathcal{Z}(M)^p$ for every $(i,j), s${\rm;} 
\item[(b)] $e_i \leq 1-1_{B_0}$, and $\mathrm{ctr}(e_i) = 2^{-k_i}c_i$ with some $k_i \in \mathbb{N}$ and $c_i \in \mathcal{Z}(M)^p$ for every $1 \leq i \leq n${\rm;}    
\item[(c)] $x_j = (1_{B_0}+\sum_{i=1}^n e_i)x_j(1_{B_0}+\sum_{i=1}^n e_i)$ for every $1\leq j\leq n'$. 
\end{itemize}
Then, for each $\varepsilon>0$ there exists a finite dimensional $*$-subalgebra $B_1$ of $M$ containing $B_0$ such that 
\begin{itemize} 
\item[(i)] $e_1,\dots,e_n \in B_1${\rm;} 
\item[(ii)] $1_{B_1} = 1_{B_0} + \sum_{i=1}^n e_i${\rm;}   
\item[(iii)] a matrix unit system $\{e_1^{(t)}(i,j)\}_{(i,j),t}$ of $B_1$ satisfies that $\mathrm{ctr}(e_1^{(t)}(i,j)) = \delta_{ij}\,2^{-k_1^{(t)}}c_1^{(t)}$ with some $k_1^{(t)} \in \mathbb{N}$ and $c_1^{(t)} \in \mathcal{Z}(M)^p$ for every $(i,j), t${\rm;} 
\item[(iv)] the distance $\mathrm{dist}_{\Vert-\Vert_\mathrm{tr}}(x_j,B_1)$ in the norm $\Vert-\Vert_\mathrm{tr}$ is less than $\varepsilon$ for every $1\leq j\leq n'$.   
\end{itemize} 
\end{lemma}

\begin{proof} Let $c_2^{(r)}$, $r\in\mathfrak{R}$,  be the (finite) partition of $(\bigvee_s c_0^{(s)})\vee(\bigvee_i c_i)$ obtained from the finite commutative family $\{c_0^{(s)}\}_s \cup \{c_i\}_i$. Define $B_0^{(r)} := B_0 c_2^{(r)}$ with matrix units $e_0^{(s,r)}(i,j) := e_0^{(s)}(i,j)c_2^{(r)}$, and also define $e_i^{(r)} := e_i c_2^{(r)}$ and $x_j^{(r)} := x_j c_2^{(r)}$. Clearly, one has $\mathrm{ctr}(e_0^{(s,r)}(i,j)) = \delta_{ij}\,2^{-k_0^{(s)}}c_2^{(r)}$ (if $e_0^{(s,r)}(i,j)\neq0$) and $\mathrm{ctr}(e_i^{(r)}) = \delta_{ij}\,2^{-k_i}c_2^{(r)}$ (if $e_i^{(r)}\neq0$). For each $r$ one can choose, by e.g.~\cite[Proposition V.1.35]{Takesaki:Book}, a sufficiently large $k_2^{(r)} \in \mathbb{N}$ and a $2^{k_2^{(r)}}\times2^{k_2^{(r)}}$ matrix unit system $f_{ij}^{(r)}$, $1 \leq i,j \leq 2^d$, in such a way that  
\begin{itemize} 
\item $\mathrm{ctr}(f_{ij}^{(r)}) = \delta_{ij}\,2^{-k_2^{(r)}}\,c_2^{(r)}$, 
\item $B_0^{(r)}\cup\{e_i^{(r)}\,|\,1\leq i \leq n\}$ sits inside the $*$-subalgebra generated by $\{f_{ij}^{(r)}\}$,  
\item $\sum_{i=1}^{2^{k_2^{(r)}}}f_{ii}^{(r)} = 1_{B_0^{(r)}} + e_1^{(r)} + \cdots + e_n^{(r)}$.   
\end{itemize}
Let us consider the$f_{1i}^{(r)}x_j^{(r)}f_{i'1}^{(r)}$, $1 \leq i,i' \leq 2^{k_2^{(r)}}$, $1 \leq j \leq n'$, inside $f_{11}^{(r)}M f_{11}^{(r)} = f_{11}^{(r)}(M c_2^{(r)})f_{11}^{(r)}$. By \cite[Theorem XVI.1.5]{Takesaki:Book} 
$f_{11}^{(r)}M f_{11}^{(r)} \cong 
R\,\bar{\otimes}\,\mathcal{Z}(f_{11}^{(r)}M f_{11}^{(r)}) = \varinjlim\,M_{2^d}(\mathbb{C})\otimes\mathcal{Z}^{(r)}_d
$ ($d\to\infty$),  
where $R$ is the hyperfinite type II$_1$ factor,  and the $\mathcal{Z}_d^{(r)}$, $d \geq 1$, form an increasing sequence of finite dimensional unital $*$-subalgebras of $\mathcal{Z}(f_{11}^{(r)}M f_{11}^{(r)})$ which generates $\mathcal{Z}(f_{11}^{(r)}M f_{11}^{(r)})$. Therefore, we can choose a unital $*$-subalgebra $C^{(r)}$ of $f_{11}^{(r)}M f_{11}^{(r)}$ such that it is isomorphic to $M_{2^d}(\mathbb{C})\,\bar{\otimes}\,\mathcal{Z}^{(r)}_d$ with some $d \in \mathbb{N}$ and that $\mathrm{dist}_{\Vert-\Vert_\mathrm{tr}}(f_{1i}^{(r)}x_j^{(r)}f_{i'1}^{(r)},C^{(r)}) < \varepsilon/(2^{2k_2^{(r)}}|\mathfrak{R}|)$ for every $1 \leq i,i' \leq 2^{k_2^{(r)}}$ and $1 \leq j \leq n'$. Let $B_1^{(r)}$ be the $*$-subalgebra generated by the $f_{ij}^{(r)}$ and $C^{(r)}$, and set $B_1 := \sum_{r\in\mathfrak{R}}^\oplus B_1^{(r)}$. Remark that $2^{k_2^{(r)}}(\mathrm{ctr}\!\upharpoonright_{f_{11}^{(r)}Mf_{11}^{(r)}})f_{11}^{(r)}$ is the unique center-valued trace on $f_{11}^{(r)}Mf_{11}^{(r)}$, and hence it agrees, via the isomorphism, with the center-valued trace $\tau_R\,\bar{\otimes}\,\mathrm{Id}$ on $R\,\bar{\otimes}\,\mathcal{Z}(f_{11}^{(r)}M f_{11}^{(r)})$. Let $g_{ij}^{(r,u)}$, $1 \leq i,j \leq 2^d$, $1 \leq u \leq \dim(\mathcal{Z}_d^{(r)})$, be a matrix unit system obtained from a standard one in $M_{2^d}(\mathbb{C})\,\bar{\otimes}\,\mathcal{Z}_d^{(r)}$ via the isomorphism. Then $2^{k_2^{(r)}}\mathrm{ctr}(g_{ij}^{(r,u)})f_{11}^{(r)} = \delta_{ij}\,2^{-d}z_u^{(r)}$ so that $z_u^{(r)} \in \mathcal{Z}(f_{11}^{(r)}M f_{11}^{(r)})^p$ with $\sum_u z_u^{(r)} = f_{11}^{(r)}$. Since $c_2^{(r)}$ is the central support of $f_{11}^{(r)}$, the mapping $x \in \mathcal{Z}(M)c_2^{(r)} \mapsto xf_{11}^{(r)} \in \mathcal{Z}(M)f_{11}^{(r)} = \mathcal{Z}(f_{11}^{(r)}M f_{11}^{(r)})$ is a bijective $*$-homomorphism, and hence there exist unique, mutually orthogonal  $c_2^{(r,u)} \in \mathcal{Z}(M)^p$, $1 \leq u \leq \dim(\mathcal{Z}_d^{(r)})$, such that $z_u^{(r)} =  c_2^{(r,u)}f_{11}^{(r)}$ for $1 \leq u \leq \dim(\mathcal{Z}_d^{(r)})$. Hence $\mathrm{ctr}(g_{ij}^{(r,u)}) = \delta_{ij}\,2^{-(k_2^{(r)}+d)}c_2^{(r,u)}$ for $1 \leq i,j \leq 2^d$ and $1 \leq u \leq \dim(\mathcal{Z}_d^{(r)})$. The matrix units $f_{ij}^{(r)}$ and $g_{ij}^{(r,u)}$ give the desired matrix unit system $e_1^{(t)}(i,j)$ of $B_1$; in fact, it is easy to see that $B_1$ satisfies the first three desired conditions. Moreover, 
$\mathrm{dist}_{\Vert-\Vert_{\mathrm{tr}}}(x_j,B_1) 
\leq 
\sum_{r \in \mathfrak{R}}\sum_{i,i'=1}^{2^{k_2^{(r)}}} \mathrm{dist}_{\Vert-\Vert_{\mathrm{tr}}}(f_{1i}^{(r)}x_jf_{i'1}^{(r)},C^{(r)}) < \varepsilon$.       
\end{proof}
 
\begin{proof} (Proposition \ref{P-a1}) Firstly assume that $M$ is of type II. Choose a countable dense subset $\{x_j\,|\,j \geq 1\}$ of the closed unit ball of $M$ equipped with the $\sigma$-strong operator topology, and we may and do assume $x_1 = 1$. Choose the smallest $k_0$ so that $p_{1,1,k_0} \neq 0$, and set $q'_n := \sum_{i\leq n} \sum_{l\leq n}\sum_{k\leq n+k_0 -1} p_{i,l,k} \nearrow \sum_{i \in I}\sum_{l\geq1}\sum_{k\geq 1} p_{i,l,k} = 1$ as $n\to\infty$. By \cite[Proposition V.1.35]{Takesaki:Book} one can find a unital copy of $M_2(\mathbb{C})$ in $q'_1 M q'_1 = p_{1,1,k_0}Mp_{1,1,k_0}$, and let $C_1$ be such a copy of $M_2(\mathbb{C})$. Assume that we have already construct $C_1 \subseteq C_2 \subseteq \cdots \subseteq C_n$ in such a way that 
\begin{itemize} 
\item[(i)] $C_{n'}$ has a matrix unit system satisfying (a) in Lemma \ref{L-a4}, 
\item[(ii)] $1_{C_{n'}} = q'_{n'}$,  
\item[(iii)] the $p_{i,l,k}$, $i,l,k \leq n'$, are in $C_{n'}$, 
\item[(iv)] $\mathrm{dist}_{\Vert-\Vert_\mathrm{tr}}(q'_{n'} x_j q'_{n'},C_{n'}) < 1/n'$, $1 \leq j \leq n'$,  
\end{itemize}
for every $1 \leq n' \leq n$. Applying Lemma \ref{L-a4} to $C_n$ ($=B_0$ in Lemma \ref{L-a4}), the family of projections $p_{i,l,k}$ with $1\leq i \leq n+1$, $0 \leq l \leq n+1$, $1 \leq k \leq n+k_0$ such that at least one of $i = n+1$, $l = n+1$, $k=n+k_0$ holds (which plays a r\^ole of $\{e_i\}$ in Lemma \ref{L-a4}), $\{q'_{n+1}x_j q'_{n+1}\,|\,1 \leq j \leq n+1\}$ ($= \{x_j\}$ in Lemma \ref{L-a4}) and $\varepsilon := 1/(n+1)$, we get a finite dimensional $*$-subalgebra $C_{n+1}$ of $M$ that is larger than $C_n$ and satisfies the properties (i)--(iv) with $n'=n+1$. By induction we have obtained an increasing sequence $C_n$ such that the properties (i)--(iv) above hold with $n'=n$ for every $n$. 

Fix an arbitrary $n \geq 1$, and choose an arbitrary $x$ from the closed unit ball of $q'_n M q'_n$. For any $\varepsilon>0$ there exists $x_{j_0}$ so that $\Vert q'_n x_{j_0} q'_n - x \Vert_\mathrm{tr} < \varepsilon/2$. For a sufficiently large $n' \in \mathbb{N}$ with $1/n' < \varepsilon/2$, $n,j_0 \leq n'$ there exists $y \in C_{n'}$ such that $\Vert q'_{n'} x_{j_0} q'_{n'} - y\Vert_\mathrm{tr} < \varepsilon/2$.  Thus $\Vert q'_n x_{j_0} q'_n - q'_n y q'_n\Vert_\mathrm{tr} < \varepsilon/2$ so that $\Vert q'_n y q'_n - x\Vert_\mathrm{tr} < \varepsilon$. Consequently, $q'_n M q'_n$ is generated by all the $q'_n C_{n'} q'_n$, ${n'} \geq n$. Since $q'_n \nearrow 1$ as $n\to\infty$, we conclude that $M$ is generated by the $C_n$.        

Let $q_n := \sum_{i\leq n} p_i$, and then $q_n - q'_n = \sum_{i\leq n} \sum_{l \geq n+1\,\text{or}\atop k \geq n+1} p_{i,l,k}$. Set $p'_i := \sum_{l \geq n+1\,\text{or}\atop k \geq n+1}p_{i,l,k}$, $1 \leq i \leq n$, and define $A_n := C_n \oplus \sum_{i\leq n}^\oplus \mathbb{C}p'_i$, which is a desired one.      

\medskip
Secondly assume that $M$ is globally of type I$_m$ with $m \in \mathbb{N}\cup\{\infty\}$. Let $C$ be a MASA in $M$ that contains $D$. By \cite[Lemma 3.7, Lemma 3.9]{Kadison:AmerJMath84} $C$ is generated by an orthogonal (finite/infinite) sequence $e_k$ of abelian projections in $M$ such that $c_M(e_k)=1$ for every $k\geq1$ and $\sum_{k\geq1} e_k = 1$. By the proof of \cite[Proposition V.1.22]{Takesaki:Book} we may and do assume that $M = B(\mathcal{H})\,\bar{\otimes}\,\mathcal{Z} \supseteq C = \Delta\,\bar{\otimes}\,\mathcal{Z}$ with a commutative von Neumann algebra $\mathcal{Z}$ and $e_k = \delta_k\otimes1$, where $\dim(\mathcal{H}) = m$, $\Delta$ is the algebra of all diagonal bounded operators with respect to a fixed basis of $\mathcal{H}$ and the $\delta_k$'s are all the minimal projections in $\Delta$. Then $p_{i,k} = e_k p_i = \delta_k\otimes c_{i,k}$ with $c_{i,k} \in \mathcal{Z}^p$. Choose an increasing sequence $\mathcal{Z}_{0,n}$ of finite dimensional unital $*$-subalgebras of $\mathcal{Z}$ which generate $\mathcal{Z}$. Define $\mathcal{Z}_n$ to be the $*$-subalgebra of $\mathcal{Z}$ generated by $\mathcal{Z}_{0,n}$ and $\{c_{i,k}\,|\,i,k \leq n\}$. Then $\mathcal{Z}_n$  and $C_n := f_n B(\mathcal{H})f_n\,\bar{\otimes}\,\mathcal{Z}_n$ with $f_n := \sum_{k\leq n}\delta_k$ are finite dimensional. Note that $q'_n := \sum_{i\leq n}\sum_{k \leq n} p_{i,k} \in C_n$. The $q'_n C_n q'_n \oplus \sum_{i \leq n}^\oplus \mathbb{C}(\sum_{k > n} p_{i,k})$, $n\geq1$, give the desired $A_n$.  

\medskip
Finally the general case is quite easy to prove now by what we have proved so far thanks to \cite[Theorem V.1.19, Theorem V.1.27]{Takesaki:Book}. 
\end{proof} 
 
\subsection{Substandard embeddings and the class $\mathcal{R}_4$}

The next definition was introduced by Redelmeier \cite{Redelmeier:arXiv:1207.1117}. The class $\mathcal{R}_4$ was already appeared in Proposition \ref{P2}. 

\begin{definition}\label{D-a1}  Let $N$ and $M$ be in the class $\mathcal{R}_4$, i.e., $N = N_h \oplus \sum_{i \in I}^\oplus N_i$ and $M = M_h \oplus \sum_{j \in J}^\oplus M_i$, where $N_h, M_h$ are hyperfinite and the $N_i, M_j$ amplifications of interpolated free group factors. An injective normal $*$-homomorphism $\pi$ from $N$ to a corner of $M$ is said to be a substandard embedding if for each $i \in I$ there exist $j \in J$ and a non-zero finite projection $p \in N_i = N1_{N_i}$ such that $\pi(1_{N_i}) \leq 1_{M_j}$, $\pi(p)$ is finite in $M$, and $\pi\!\upharpoonright_{pNp} : pNp = p N_i p \to \pi(p)M\pi(p) = \pi(p)M_j\pi(p)$ is a standard embedding in the sense of Dykema \cite[Definition 4.1]{Dykema:DukeMathJ93} {\rm(}allowed to be an identity mapping{\rm)}. \end{definition} 

Remark that $j$ is determined uniquely from a given $i$ in the above definition. Thus $\pi$ gives a well-defined mapping $i \mapsto \pi_*(i)$ from $I$ into $J$. Note that $1_{N_i} \leq 1_{M_{\pi_*(i)}}$ for every $i \in I$, and hence $\sum_{i\in \pi_*^{-1}(\{j\})} 1_{N_i} \leq 1_{M_j}$ for every $j \in J$. In particular, $z_f^N := \sum_{i \in I} 1_{N_i} \leq z_f^M := \sum_{j \in J} 1_{M_j}$. The next proposition and its proof have never been given at least explicitly before. 

\begin{proposition}\label{P-a5} Let $M_n$, $n \geq 1$, be an increasing sequence of semifinite von Neumann subalgebras of a semifinite von Neumann algebra $M$ equipped with a faithful normal semifinite trace $\mathrm{tr}$ such that 
\begin{itemize} 
\item[(a)] $\mathrm{tr}\!\upharpoonright_{M_n}$ is semifinite for every $n$, 
\item[(b)] The $M_n$ generate $M$ as von Neumann algebra, 
\item[(c)] every $M_n$ is in the class $\mathcal{R}_4$, that is, $M_n = M_{n,h} \oplus \sum_{i \in I_n}^\oplus M_{n,i}$ as in Definition \ref{D-a1}, 
\item[(d)] every inclusion $M_n \subseteq M_{n+1}$ is a substandard embedding. 
\end{itemize}  
Then $M$ must fall in the class $\mathcal{R}_4$.  
\end{proposition}
\begin{proof} Set $N_n := M_n + \mathbb{C}(1 - 1_{M_n})$ for every $n$. As remarked just before this proposition we have, by the assumptions (c),(d), $z_f^{(n)} := \sum_{i \in I_n} 1_{M_{n,i}} \leq z_f^{(n+1)} := \sum_{i \in I_{n+1}} 1_{M_{n+1,i}}$ for every $n$. Hence the central support projection $z_h^{(n)}$ of the hyperfinite part $M_{n,h} \oplus \mathbb{C}(1 - 1_{M_n})$ of $N_n$ must satisfy that $z_h^{(n)} = 1 - z_f^{(n)} \geq 1 - z_f^{(n+1)} = z_h^{(n+1)}$. Consequently, $z_h := \lim_{n\to\infty}z_h^{(n)}$ exists and falls in $\mathcal{Z}(M)$, since the $N_n$ are not decreasing and generate $M$ by the assumption (b). Every $N_n z_h = N_n z_h^{(n)} z_h$ is hyperfinite; hence $Mz_h$  must be hyperfinite. Note that $\bigvee_n \bigvee_{i \in I_n} c_M(1_{M_{n,i}}) \geq \bigvee_n \bigvee_{i \in I_n} 1_{M_{n,i}} = \bigvee_n z_f^{(n)} = 1-z_h$, a central projection of $M$. For every $n$ and $i \in I_n$ one has $1_{M_{n,i}} \leq z_f^{(n)} \nearrow 1 - z_h$, implying $c_M(1_{M_{n,i}}) \leq 1 - z_h$. Consequently, $\bigvee_n \bigvee_{i \in I_n} c_M(1_{M_{n,i}}) = 1 - z_h$. Let us choose arbitrary $n$ and $i \in I_n$. By the assumption (d) there exist a unique path $i=i(n)\to i(n+1)\to\cdots$ with $i(k) \in I_k$ for every $k \geq n$ and a sequence of finite projections $p_k \in M_{k,i(k)} = M1_{M_{k,i(k)}}$, $k \geq n$, such that the $1_{M_{k,i(k)}}$ are not decreasing along $k \geq n$, every $p_k$ is finite in $M_{k+1,i(k+1)} = M1_{M_{k+1,i(k+1)}}$, and moreover every inclusion $p_k M_k p_k \hookrightarrow p_k M_{k+1,i(k+1)} p_k$ is a standard embeddings. By the assumption (a) all the $\mathrm{tr}(p_k)$ are finite, and thus the inductive use of \cite[Proposition 4.2]{Dykema:DukeMathJ93} shows that all the inclusions $p_n M_{n,i} p_n \subseteq p_n M_{n+1,i(n+1)}p_n \subseteq p_n M_{n+2,i(n+1)}p_n \subseteq \cdots$ are standard embeddings. Consequently, \cite[Proposition 4.3 (ii)]{Dykema:DukeMathJ93} shows that $p_n Mp_n$ (which is generated by the $p_n M_{k,i(k)} p_n = p_n M_k p_n$, $k \geq n$) must be an interpolated free group factor so that $Mc_M(1_{M_{n,i}})$ is an amplification of that interpolated free group factor. Remark that all the $1_{M_{k,i(k)}}$, $k \geq n$, are in $Mc_M(1_{M_{n,i}})$, and thus $c_M(1_{M_{k,i(k)}})$, $k \geq n$, must coincide, since $Mc_M(z_{M_{n,i}})$ is a factor. Remark also that $1_{M_{k,i(k)}}$ increasingly converges to a central projection, say $c(n,i)$, of $M$, since $1_{M_{k,i(k)}} \in (N_k)'\cap M$ and the $N_k$, $k \geq n$, are not decreasing and generate $M$. This immediately implies that $1_{M_{k,i(k)}} \nearrow c(n,i) = c_M(1_{M_{n,i}})$ as $k\to\infty$ along $k \geq n$. Assume that $c_M(1_{M_{n,i}}) c_M(1_{M_{n',i'}}) \neq 0$. Since $c_M(1_{M_{n,i}}) c_M(1_{M_{n',i'}}) = \lim_{k,l\to\infty} 1_{M_{k,i(k)}} 1_{M_{l,i'(l)}}$ as remarked just before, one has $1_{M_{k,i(k)}} 1_{M_{l,i'(l)}} \neq 0$ for some $k,l$. Then $1_{M_{k\vee l,i(k\vee l)}} 1_{M_{k\vee l,i'(k\vee l)}} \geq 1_{M_{k,i(k)}} 1_{M_{l,i'(l)}} \neq 0$, which means that the unique paths $i=i(n)\to i(n+1)\to\cdots$ and $i'=i'(n')\to i'(n'+1)\to\dots$ intersect at $k\vee l$ so that $i(m) = i'(m)$ for every $m \geq k\vee l$. Therefore, $c_M(1_{M_{n,i}}) = \lim_{m\to\infty} 1_{M_{m,i(m)}} = \lim_{m\to\infty} 1_{M_{m,i'(m)}} = c_M(1_{M_{n',i'}})$ along $m \geq k\vee l$. Hence we can choose an orthogonal family $c_M(1_{M_{n_j,i_j}})$, $j\geq 1$, so that $\sum_j c_M(1_{M_{n_j,i_j}}) = 1 - z_h$.  
\end{proof} 

The next lemma is elementary so that we leave its proof to the reader. 

\begin{lemma}\label{L-a6} Let $M$ be a von Neumann algebra and $e \in M^p$ be such that $c_M(e) = 1$. If $eMe$ falls into the class $\mathcal{R}_4$, then so does $M$ itself. Moreover, each amplification of interpolated free group factor in $eMe$ is a corner of one in $M$.  
\end{lemma}

\subsection{A proof of Proposition \ref{P2}} 

Let us begin with the next user-friendly lemma that reorganize several arguments in \cite{Dykema:AmerJMath95} as a single statement.     

\begin{lemma}\label{L-a7} Let $B \supseteq A \supseteq D \subseteq C$ be {\rm(}unital{\rm)} finite von Neumann algebras equipped with faithful normal tracial states $\tau_B,\tau_C$ such that $\tau_B\!\upharpoonright_D = 
\tau_C\!\upharpoonright_D$. Let $z \in \mathcal{Z}(A)\cap\mathcal{Z}(B)$ be a projection such that $Bz$ is finite dimensional and $Bz^\perp = Az^\perp$. Consider two amalgamated free products $N := A \star_D C \subseteq M := B\star_D C$ with respect to the conditional expectations determined by $\tau_B, \tau_C$. If $N$ falls in the class $\mathcal{R}_4$, then so does $M$ and the embedding $N \hookrightarrow M$ is substandard. 
\end{lemma}
\begin{proof}
Examining the Bratteli diagram of $Az \subseteq Bz$ we can divide $A \subseteq B$ into $A = A_0 \subseteq A_1 \subseteq \cdots \subseteq A_l = B$ (with finite $l$) such that each $A_{k-1} \subseteq A_k$ is conjugate to either 
\begin{itemize} 
\item[(a)] $M_n(\mathbb{C})\,\oplus\,Q \hookrightarrow \big[M_n(\mathbb{C})\,\otimes\,\mathbb{C}^m\big]\,\oplus\,Q$ by $(x_1,y) \mapsto (x_1\otimes1,y)$ or, 
\item[(b)] $\big[M_{n_1}(\mathbb{C})\oplus M_{n_2}(\mathbb{C})\big]\,\oplus\,Q \hookrightarrow M_{n_1+n_2}(\mathbb{C})\,\oplus\,Q$ by 
$(x_1,x_2,y) \mapsto \big(\mathrm{Diag}(x_1,x_2),y)$. 
\end{itemize}
By definition it immediately follows from \cite[Proposition 4.2, Proposition 4.3 (i)]{Dykema:DukeMathJ93} that the composition of given substandard embeddings is substandard. Hence we may and do assume that $A \subseteq B$ is either the above (a) or (b) with letting $p := 1_Q^\perp$ in what follows.    
 
Case (a): Choose a minimal projection $e$ in $Ap$. By \cite[Lemma 4.2]{Dykema:BLMS11} (a reformulation of \cite[Lemma 4.4]{Dykema:AmerJMath95}) one has $c := c_M(e) = c_N(e) \geq p$ and $eNe \subseteq eMe$ is conjugate to $eNe \subseteq eNe \star \mathbb{C}^m$. Hence \cite[Theorem 3.2, Lemma 3.7]{Dykema:AmerJMath95} shows that $eMe$ falls in the class $\mathcal{R}_4$ and the embedding $eNe \hookrightarrow eMe$ is substandard. Then $M = Nc^\perp\oplus Mc$ and Lemma \ref{L-a6} enable us to show that $M$ falls in the class $\mathcal{R}_4$ and the embedding $N \hookrightarrow M$ is substandard. 

Case (b). Let $v$ be a matrix unit in $M_{n_1+n_2}(\mathbb{C}) = Bp$  such that $e:=v^*v \in M_{n_1}(\mathbb{C})\oplus\{0\}$ and $f:=vv^* \in \{0\}\oplus M_{n_2}(\mathbb{C})$; hence $e,f \in A$, $e \perp f$, $q:=e+f \leq p$ and both $e,f$ are minimal in $B$. Clearly $A$ and $v$ generate $B$ so that $M = N\vee\{v\}''$. We can prove, by free etymology technique (see the first paragraph of \cite[p.158]{Dykema:AmerJMath95}), that $qNq$ and the linear span of $\{e,f,v,v^*\}$ ($\cong M_2(\mathbb{C})$) are freely independent in the $\mathbb{C}e\oplus\mathbb{C}f$-valued probability space $(qMq, E_A^B\circ E_B\!\upharpoonright_{qMq})$ and generate $qMq$, where $E_A^B : B \to A$ and $E_B : M \to B$ are the unique conditional expectations determined by the tracial states that we are employing. Therefore, $qNq \hookrightarrow qMq$ is conjugate to $qNq \hookrightarrow (qNq \supseteq \mathbb{C}e\oplus\mathbb{C}f) \star_{\mathbb{C}^2} (M_2(\mathbb{C}) \supseteq \Delta_2)$, where $\Delta_2$ is the diagonals in $M_2(\mathbb{C})$. If either $e$ or $f$ is minimal and central in $qNq$, then $fMf = fNf$ or $eMe = eNe$ holds respectively, and hence by Lemma \ref{L-a6} $qMq$ falls in the class $\mathcal{R}_4$ and the embedding $qNq \hookrightarrow qMq$ is substandard. If neither $e$ nor $f$ is minimal and central, then \cite[Lemma 4.2]{Dykema:AmerJMath95} guarantees that the same assertions hold ture. Since $M = N\vee Bp$, $c_M(q) = c_M(p)$ and $c_N(q) = c_N(p) \geq p$, one has $c_M(q) = c_N(q)$. Hence the desired assertion follows as in the case (a).  
\end{proof}

\begin{lemma}\label{L-a8} Any semifinite tracial amalgamated free product of two hyperfinite von Neumann algebras over an atomic type I von Neumann subalgebra falls in the class $\mathcal{R}_4$.
\end{lemma}  
\begin{proof} Let $M_1, M_2$ be two hyperfinite von Neumann algebras and $D$ be a common atomic type I von Neumann subalgebra. Assume that there exist two faithful normal conditional expectations $E_k : M_k \to D$, $k=1,2$, such that $\psi\circ E_k$, $k=1,2$, give faithful normal (semifinite) traces $\mathrm{tr}_k$ on $M_k$ for a common faithful normal (semifinite) trace $\psi$ on $D$. Let $(M,E) = (M_1,E_1)\star_D (M_2,E_2)$ be the amalgamated free product, and by \cite[Theorem 2.6]{Ueda:PacificJMath99} the composition $\psi\circ E$ gives a faithful normal (semifinite) trace $\mathrm{tr}$ on $M$. By Lemma \ref{L-a6} and by e.g.~\cite[Lemma 4.6]{Ueda:JLMS13} we may and do assume that $D$ is commutative, say $D = \sum^\oplus_{i\geq1} \mathbb{C}p_i$. (This observation originates in \cite[Lemma 2.2]{Radulescu:InventMath94}.) By Proposition \ref{P-a1} with $q_n := \sum_{i\leq n} p_i \in D$ we have two sequences of finite dimensional $*$-subalgebras $A_{k,1} \subseteq A_{k,2} \subseteq \cdots$ of $M_k$, $k=1,2$, such that $D q_n \subseteq A_{k,n}$ and $q_n = 1_{A_{k,n}}$, $k=1,2$, for every $n$ and all $A_{k,n}$, $n\geq1$, generate $M_k$, $k=1,2$. Set $P_n := A_{1,n}\vee A_{2,n}$, $n \geq 1$, inside $q_n Mq_n$. The $P_n$, $n\geq1$, form an increasing sequence and generate $M$. We will prove by induction that all the $P_n$ fall in the class $\mathcal{R}_4$ and all the embeddings $P_n \hookrightarrow P_{n+1}$ are substandard. If once these were established, then the desired assertion would immediately follow by Proposition \ref{P-a5}. Since $P_1 \cong A_{1,1}\star A_{2,1}$, it must fall in the class $\mathcal{R}_4$ by \cite[Theorem 3.6]{Dykema:DukeMathJ93}. Assume that we have proved that all the $P_k$, $k \leq n$ fall in the class $\mathcal{R}_4$ and that all the embeddings $P_k \hookrightarrow P_{k+1}$, $k \leq n-1$, are substandard. Set $B_{k,n+1} := A_{k,n} \oplus\,\mathbb{C}p_{n+1}$, a unital $*$-subalgebra of $A_{k,n+1}$ for $k=1,2$. Consider $Q := B_{1,n+1}\vee B_{2,n+1} \subseteq R := A_{1,n+1}\vee B_{2,n+1}$. Then $Q = P_n\oplus\mathbb{C}p_{n+1}$, and $Q \subseteq R \subseteq P_{n+1}$ is conjugate to $B_{1,n+1}\star_{Dq_{n+1}} B_{2,n+1} \subseteq A_{1,n+1}\star_{Dq_{n+1}} B_{2,n+1} \subseteq A_{1,n+1}\star_{Dq_{n+1}} A_{2,n+1}$. Hence using Lemma \ref{L-a7} twice we see that $Q,R,P_{n+1}$ fall in the class $\mathcal{R}_4$ and the embeddings $P_n \hookrightarrow Q \hookrightarrow R \hookrightarrow P_{n+1}$ are substandard. As remarked in the proof of Lemma \ref{L-a7} the composition $P_n \hookrightarrow P_{n+1}$ becomes a substandard embedding. 
\end{proof} 

The next semifinite variant of \cite[Proposition 2.2]{Dykema:PacificJMath94} may be known among specialists, but we could not find its reference; hence we include it for providing its clear statement.    

\begin{lemma}\label{L-a9} Let $F$ be an amplification of interpolated free group factor and $p_i$, $i \geq 1$, be a partition of unity of projections in $F$. Assume that a faithful normal {\rm(}semifinite{\rm)} trace $\mathrm{tr}_F$ on $F$ satisfies that $\mathrm{tr}_F(p_i) < +\infty$ for every $i \geq 1$. Then there exist a semifinite von Neumann algebra $M$ with a faithful normal {\rm(}semifinite{\rm)} trace $\mathrm{tr}_M$, a copy of hyperfinite type II$_1$ or II$_\infty$ factor $R$ in $M$, a partition of unity $q_i$, $i \geq 1$, of projections in $R$, a semicircular family $\{x_t\,|\,t \in T\}$ in $(q_1 M q_1,\mathrm{tr}_M(q_1)^{-1}\mathrm{tr}_M\!\upharpoonright_{q_1 M q_1})$ and projections $\{e_t\,|\,t \in T\}$ in $q_1 R q_1$ such that $\mathrm{tr}_M(q_i) = \mathrm{tr}_F(p_i)$ for every $i \geq 1$, that $q_1 R q_1$ and $\{x_t\,|\, t \in T\}$ are freely independent in $(q_1 M q_1,\mathrm{tr}_M(q_1)^{-1}\mathrm{tr}_M\!\upharpoonright_{q_1 M q_1})$, and that $F$ is isomorphic to $R \vee \{e_t x_t e_t\,|\, t \in T\}''$ with sending $p_i$ to $q_i$ for every $i \geq 1$,
\end{lemma} 
\begin{proof} By assumption one has $r > 1$ so that $p_1 F p_1 \cong L(\mathbb{F}_r)$. Let $R$ be the unique hyperfinite type II$_1$ or II$_\infty$ factor with a faithful normal (semifinite) trace $\mathrm{tr}$. One can find a partition of unity $q_i$, $i \geq 1$, of projections in $R$ so that $\mathrm{tr}(q_i) = \mathrm{tr}_F(p_i)$ for every $i \geq 1$. Set $D := \sum_{i\geq1}^\oplus\mathbb{C}q_i \subseteq R$. Consider the amalgamated free product von Neumann algebra
$M = R \star_D \big[\overset{q_1}{L(\mathbb{F}_\infty)}\oplus Dq_1^\perp\big]$
with respect to the conditional expectations determined by $\mathrm{tr}$ and the tracial state $\tau_{\mathbb{F}_\infty}$ on $L(\mathbb{F}_\infty)$. Then $M$ has a faithful normal (semifinite) trace $\mathrm{tr}_M$ with $\mathrm{tr}_M\!\upharpoonright_{R} = \mathrm{tr}$. (See the proof of Lemma \ref{L-a8}.) Choose projections $\{e_t\,|\,t \in T\}$ in $q_1 Rq_1$ in such a way that $\sum_{t \in T} (\mathrm{tr}(e_t)/\mathrm{tr}(q_1))^2 = r-1$, and the copy of $L(\mathbb{F}_\infty)$ in $M$ has a semicircular family $\{x_t\,|\,t \in T\}$ with respect to $\tau_{\mathbb{F}_\infty}$. Set $N := R\vee\{e_t x_t e_t\,|\,t \in T\}''$, which has a faithful normal (semifinite) trace $\mathrm{tr}_N := \mathrm{tr}_M\!\upharpoonright_N$. Then $q_1 N q_1 = q_1 R q_1 \vee \{e_t x_t e_t\,|\,t \in T\}'' \cong p_1 F p_1$ by Dykema's definition of interpolated free group factors, and $c_N(q_1) = 1_N$. By \cite[Proposition V.1.40]{Takesaki:Book} we conclude $N \cong F$ that preserves the traces $\mathrm{tr}_N, \mathrm{tr}_F$. Perturbing the $*$-isomorphism by a unitary in $F$ one gets a $*$-isomorphism $\pi : N \to F$ so that $\mathrm{tr}_F\circ\pi = \mathrm{tr}_N$ and $\pi(q_i) = p_i$ for every $i \geq 1$. 
\end{proof}  

The next lemma is related to \cite[\S4]{Redelmeier:arXiv:1207.1117} that has a problem at the present moment. The proof below does not yet work for fixing it completely.     

\begin{lemma}\label{L-a10} Let $M_1, M_2$ be either {\rm(a)} a hyperfinite von Neumann algebra and an amplification of interpolated free group factor, or {\rm(b)} both amplifications of interpolated free group factors. Let $D = \sum_{i\geq1}^\oplus\mathbb{C}p_i$ be a common unital von Neumann subalgebra of the $M_k$, $k=1,2$. Assume that both the $M_k$ have faithful normal {\rm(}semifinite{\rm)} traces $\mathrm{tr}_k$ which are semifinite on $D$ and satisfy $\mathrm{tr}_1\!\upharpoonright_D = \mathrm{tr}_2\!\upharpoonright_D$. Then the amalgamated free product von Neumann algebra $M := M_1 \star_D M_2$ with respect to the conditional expectations determined by the traces $\mathrm{tr}_k$ is an amplification of interpolated free group factor again and each embedding $M_k \hookrightarrow M$ is substandard if $M_k$ is an amplification of interpolated free group factor. 
\end{lemma}
\begin{proof} We first provide a preliminary fact when $M_1, M_2$ are hyperfinite type II$_1$ or II$_\infty$ factors instead. By Proposition \ref{P-a1} one can choose a sequence of non-trivial finite dimensional $*$-subalgebras $A_{2,1} \subseteq A_{2,2} \subseteq \cdots$ of $M_2$ such that $q_n : = \sum_{i \leq n} p_i \nearrow 1$ as $n\to\infty$, $q_n = 1_{A_{2,n}} \in D$, $Dq_n \subseteq A_{2,n}$ for every $n \in \mathbb{N}$, and all the $A_{2,n}$ generate $M_2$. Set $B_{2,n} := A_{2,n}\oplus Dq_n^\perp$, and consider the von Neumann subalgebras $P_n := M_1 \vee B_{2,n} \cong M_1 \star_D B_{2,n}$. Then $p_1 P_1 p_1 = p_1 M_1 p_1 \vee A_{21} \cong p_1 M_1 p_1 \star A_{21}$ becomes an interpolated free group factor by Dykema's result \cite[Theorem 4.6]{Dykema:DukeMathJ93}, and moreover by the proofs of \cite[Proposition 4.4 (ii)]{Dykema:DukeMathJ93} and \cite[Theorem 4.1]{Dykema:PacificJMath94} $p_1 M_1 p_1 \hookrightarrow p_1 P_1 p_1$ is a standard embedding.  Remark here that \cite[Definition 4.1]{Dykema:DukeMathJ93} makes sense even when the initial algebra $L(\mathbb{F}_r)$ there is just the hyperfinite II$_1$ factor. By the construction of the $B_{2,n}$ one has $q_{n+1} P_n q_{n+1} \cong q_{n+1} M_1 q_{n+1} \star_{Dq_{n+1}} \big(A_{2,n}\oplus\mathbb{C}p_{n+1}\big)$ and $q_{n+1} P_{n+1} q_{n+1} \cong q_{n+1} M_1 q_{n+1} \star_{Dq_{n+1}} A_{2,n+1}$ simultaneously. Hence a standard induction argument based on Lemma \ref{L-a7} together with \cite[Proposition 4.2]{Dykema:DukeMathJ93} shows that every embedding $p_1 P_n p_1 \hookrightarrow p_1 P_{n+1}p_1$ is standard, where one naturally observes that all the $P_n$ are amplifications of interpolated free group factors. By \cite[Proposition 4.3]{Dykema:DukeMathJ93} we conclude that $M$ is an amplification of interpolated free group factor and that the embedding $p_1 M_1 p_1 \hookrightarrow p_1 M p_1$ is standard. The last fact implies the following observation: By \cite[Proposition V.1.40]{Takesaki:Book} one can choose a (possibly infinite size) matrix unit system $e_{ij}$ in $M_1$ in such a way that $e_{11} = p_1$. By the definition of standard embeddings together with \cite[Proposition 2.2]{Dykema:PacificJMath94} one has $p_1 M p_1 = p_1 M_1 p_1 \vee \{e_t x_t e_t\,|\,t \in T\}''$, where $\{x_t\,|t\in T\}$ is a semicircular system and freely independent of $p_1 M_1 p_1$ in a bigger tracial $W^*$-probability space containing $(M,\mathrm{tr}_M(p_1)^{-1}\mathrm{tr}\!\,\upharpoonright_{p_1 M p_1})$ with the canonical trace $\mathrm{tr}_M$ on $M$ and $\{e_t\,|\,t \in T\}$ are projections in $p_1 M_1 p_1$. Thus $M = p_1 M_1 p_1 \vee \{e_t x_t e_t\,|\,t \in T\}'' \vee \{e_{ij}\}'' = M_1 \vee \{e_t x_t e_t\,|\,t \in T\}''$. 

Case (a): The same pattern of argument as above apparently works. 

Case (b): By Lemma \ref{L-a9} we can assume that $D \subseteq R_k \subseteq M_k = R_k \vee \{e_t^{(k)} x_t^{(k)} e_t^{(k)}\,|\,t \in T^{(k)}\}''$, $k=1,2$, as there, that is, all $e_t^{(k)} $ are in $p_1 R_k p_1$. Then $p_1 M p_1 = p_1(R_1 \vee R_2)p_1 \vee \{e_t^{(1)}x_t^{(1)}e_t^{(1)}\,|\,t \in T^{(1)}\}'' \vee \{e_t^{(2)}x_t^{(2)}e_t^{(2)}\,|\,t \in T^{(2)}\}''$ which can be re-written as 
$p_1 R_1 p_1 \vee \{e_t x_t e_t\,|\,t \in T\}'' \vee \{e_t^{(1)}x_t^{(1)}e_t^{(1)}\,|\,t \in T^{(1)}\}'' \vee \{e_t^{(2)}x_t^{(2)}e_t^{(2)}\,|\,t \in T^{(2)}\}''$ by the observation provided in the first paragraph. By the same reasoning as in the proof of \cite[Proposition 4.3 (i)]{Dykema:DukeMathJ93} we may and do assume that $\{x_t\,|\,t \in T\} \cup 
\{x_t^{(1)}\,|\,t \in T^{(1)}\}\cup\{x_t^{(2)}\,|\,t \in T^{(2)}\}$ is a  semicircular system that is freely independent of $p_1 R_1 p_1$ (in a bigger finite tracial $W^*$-probability space with unit $p_1$) and all the $e_t^{(2)}$ are in $p_1 R_1 p_1$ with keeping the $e_t^{(1)}x_t^{(1)}e_t^{(1)}$. It is now trivial that $p_1 M_1 p_1 \hookrightarrow p_1 M p_1$ is a standard embedding so that $M_1 \hookrightarrow M$ itself a substandard embedding.     
\end{proof} 

\begin{proof} (Proposition \ref{P2}) 
As Lemma \ref{L-a8} it suffices to investigate a semifinite tracial amalgamated free product von Neumann algebra $M = M_1 \star_D M_2$ over atomic \emph{commutative} $D$. 

Firstly, assume that $M_1$ is in the class $\mathcal{R}_4$ and $M_2$ an amplification of interpolated free group factor. We will prove that $M$ is an amplification of interpolated free group factor and $M_1 \hookrightarrow M$ is a substandard embedding. Write $M_1 = M_{1,h}\oplus\sum_{i \in I}^\oplus M_{1,i}$ as in Definition \ref{D-a1}. Let $I = \{1,2,\dots\}$ be an arbitrary enumeration. It suffices to prove that the embedding $M_{1,1} \hookrightarrow M$ is substandard. Denote $N_k := M_{1,h}\oplus\sum_{i\leq k}^\oplus M_{1,i} \oplus D(\sum_{i\geq k+1} 1_{M_{1,i}})$, $k\geq 0$, and set $P_k := N_k \vee M_2 \cong N_k \star_D M_2$. By Lemma \ref{L-a10} (a), $P_0$ is an amplification of interpolated free group factor. By Dykema's trick, see e.g.~\cite[Lemma 4.3]{Dykema:AmerJMath95} ({\it n.b.}~it holds in the non-tracial setting too with the same proof), we have $1_{M_{1,1}}P_1 1_{M_{1,1}} = M_{1,1} \star_{D1_{M_{1,1}}} 1_{M_{1,1}}P_0 1_{M_{1,1}}$ and $c_{P_1}(1_{M_{1,1}}) = c_{P_0}(1_{M_{1,1}}) = 1$. Since $M_{1,1}$ and $P_0$ are amplifications of interpolated free group factors, Lemma \ref{L-a10} (b) shows that so is $P_1$ and the embedding $M_{1,1} \hookrightarrow P_1$ is substandard. This procedure can be continued, and hence we get a sequence of substandard embeddings $M_{1,1} \hookrightarrow P_1 \hookrightarrow P_2 \hookrightarrow \cdots$ of amplifications of interpolated free group factors such that the $P_k$ generate $M$. Using \cite[Proposition 4.2]{Dykema:DukeMathJ93} inductively we can find a finite projection $e \in M_{11}$  so that $e M_{11} e \hookrightarrow e P_1 e \hookrightarrow \cdots$ inside $e M e$ form a sequence of standard embeddings. The proof of \cite[Proposition 4.3 (ii)]{Dykema:DukeMathJ93} shows that $eMe$ is an interpolated free group factor and the embedding $e M_{11} e \hookrightarrow e M e$ is standard.    

The desired assertion is proved by induction (at most) twice on the numbers of amplifications of interpolated free group factors firstly in $M_1$ and then in $M_2$; thus one needs to use Proposition \ref{P-a5} (at most) twice. The first step of the first induction was proved as Lemma \ref{L-a8}. Then each step in the inductions can be formulated as follows. Assume that $M_1 = M_{10}\oplus M_{11}$ such that $N := (M_{10}\oplus D1_{M_{11}})\star_D M_2$ falls in the class $\mathcal{R}_4$ and $M_{11}$ is an amplification of interpolated free group factor. As before we have $1_{M_{11}}M1_{M_{11}} = 1_{M_{11}}N1_{M_{11}} \star_{D1_{M_{11}}} M_{11}$ and $c := c_N(1_{M_{11}}) = c_M(1_{M_{11}})$. Note that $M = Nc^\perp \oplus Mc$, and what we proved above with Lemma \ref{L-a6} shows that $M$ falls in the class $\mathcal{R}_4$ and the embedding $N \hookrightarrow M$ is substandard. 
\end{proof}

\section*{Acknowledgments} We thank Ken Dykema and Daniel Redelmeier for electronic correspondences. We also thank Cyril Houdayer for some discussions at IHP in 2011 and his helpful comments on \S\S2.4.
}

\end{document}